\documentclass[12pt]{article}
\date{}
\usepackage{amsmath,mathrsfs,amsthm,amssymb}
\usepackage{geometry}
\usepackage{graphicx}
\usepackage{gensymb}
\usepackage{wrapfig}
\usepackage{relsize,tikz-cd}
\usepackage{amsfonts}
\usepackage[utf8]{inputenc}
\usepackage[T1]{fontenc}

\usepackage{subcaption,cleveref}
\usepackage{commath}
\usepackage[none]{hyphenat}
\usepackage{url}
\makeatletter
\newcommand{\address}[1]{\gdef\@address{#1}}
\newcommand{\email}[1]{\gdef\@email{\url{#1}}}
\newcommand{\@endstuff}{\par\vspace{\baselineskip}\noindent\small
\begin{tabular}{@{}l}\scshape\@address\\\textit{E-mail address:} \@email\end{tabular}}
\AtEndDocument{\@endstuff}
\makeatother

\title{Invariants of annular links, cobordisms and transverse links from combinatorial link Floer complex}

\author{Apratim Chakraborty}
\address{Stat-Math Unit, Indian Statistical Institute, Bangalore Centre,\\  8th Mile Mysore Road, Bangalore, 560059}
\email{apratimn@gmail.com}

\newtheorem{prop}{Proposition}[section]
\newtheorem{lemma}{Lemma}[section]
\newtheorem{proposition}{Proposition}[section]
\newtheorem{theorem}{Theorem}[section]

\newtheorem{cor}{Corollary}[section]

\newtheorem{definition}{Definition}[section]
\newtheorem*{remark}{Remark}
\begin{document}    
\maketitle

\begin{abstract}

We define an annular concordance invariant and study its properties. When specialized to braids, this invariant gives bounds on band rank. We introduce a modified chain complex to reformulate the invariant. Then, by focusing on a special case, we give a refinement of the transverse invariant $\hat{\theta}$. We also study the relationship of this invariant with transverse and braid monodromy properties.   

\end{abstract}
\section{Introduction}
An annular link $L$ can be thought of as a link sitting inside a thickened annulus $ A \times I$. Asaeda-Przytycki-Sikora \cite{asaeda} and L.Roberts \cite{lroberts} defined a version of Khovanov homology for annular links which is known as annular Khovanov complex (or sutured Khovanov complex). In the annular Khovanov complex, one can extract an extra filtration for annular links. Grigsby, Wehrli and Licata \cite{annular} made use of $\mathbb{Z} \oplus \mathbb{Z}$ filtered nature of annular Khovanov complex to define an annular link invariant. The invariant is a piece-wise linear function which gives genus bounds for annular cobordisms. \\

In the same vein, we will define an invariant $\mathscr{A}_{L}(t)$ (for $t \in [0,2]$) of annular links using link Floer complex. Recently, Truong and Zhang \cite{zhang} also introduced a similar invariant in Sarkar-Seed-Szabo\'{o} complex. In our case, we will think of the annular link as an union $U\cup L$. We will define the invariant using Alexander filtrations of the the filtered complex $\hat{\mathcal{GC}}$. Assume $\mathcal{F}_{t}$ is the weighted filtration  $\frac{t}{2}A_{U}+(1-\frac{t}{2})A_{L}$ where $A_{U}$ and $A_{L}$ are Alexander filtrations for $U$ and $L$ respectively. Then, $\mathscr{A}_{L}(t)$ is defined to be the minimum $\mathcal{F}_{t}$ filtration level where homology at Maslov grading $0$ is non-trivial. The invariant has properties similar to Grigsby, Wehrli and Licata invariant. It gives a lower bound on annular cobordisms.

\begin{theorem} \label{theorem1}
If $\Sigma$ is a strong annular cobordism of genus $g$ between two annular links $L_{1}$ and $L_{2}$  then
$| \mathscr{A}_{L_{1}}(t)-\mathscr{A}_{L_{2}}(t) | \leq g(1-\frac{t}{2})$.
\end{theorem}

In a different direction, Baldwin, Vela-Vick and V\'{e}rtesi \cite{lossbraid} showed the equivalence of LOSS and GRID transverse invariants for transverse knots in $\mathbb{S}^{3}$ with the standard contact structure. Their work involved passing to the knot Floer complex $CFK^{-,2}(-U \cup \beta)$ (Here $\beta$ is a braid representing a transverse link) which comes with an extra filtration $\mathcal{F}^{-U}$. A key proposition in their work showed that the distinguished class $[x_{4}]$ generated $H_{top}(\mathcal{F}^{-U}_{bot})$. In view of that, it is natural to ask if the filtered quasi-isomorphism itself contains any information about the transverse knot represented by $\beta$. From a different perspective, we define a complex ${\mathcal{C}}_{U \cup \beta}$  turns out to be isomorphic to $CFK^{-,2}(-U \cup \beta)$. We can independently show that filtered quasi-isomorphism type of that complex is an invariant of the braid conjugacy class. Moreover, by studying the crossing change moves we study how it changes under positive stabilization. We can then extract a numerical invariant $\eta(\beta)$ from the complex that is similar to the braid conjugacy class invariant $\kappa$ defined by Hubbard and Saltz \cite{saltz} in Khovanov homology.

\begin{theorem}\label{theorem2}
$\eta(\beta)$ is a braid conjugacy class invariant. Also $ \eta(\beta) \leq \eta(\beta_{+stab}) \leq \eta(\beta)+\frac{1}{2}$. 
\end{theorem}

Since it only increases under positive stabilization, one can get a transverse invariant by taking minimum over all braid representatives. We also prove the following relationship with the $\theta$ invariant.

\begin{theorem}\label{theorem3}
Let $\beta$ be a $N$ braid and $\mathcal{T}$ be the transverse link represented by $\beta$. If $\hat{\theta}(\mathcal{T})\neq 0$ then $\eta(\beta)=\infty$ and  $-\frac{N}{2} \leq \eta(\beta)\leq \frac{N}{2}$ otherwise.
\end{theorem}
The above theorem shows that $\eta$ is a refinement of the transverse invariant $\hat{\theta}$.\\

For an alternate formulation of $\mathscr{A}_{L}(t)$, we define a modification $t\mathbf{C}$ of the usual $GC^{-}$ grid link complex for every $t \in [0,2]$. For $t=0$, we recover the complex ${\mathcal{C}}_{U \cup \beta}$. We prove that $\mathscr{A}_{L}(t)$ can also be interpreted as a max grading of a non-torsion element $\alpha$ in a t-modified complex $t\mathbf{C}$. 

\begin{theorem}\label{theorem4}
 $\mathcal{A}_{L}(t) = - \mathscr{F}_{t}([\alpha])$. 
\end{theorem}

We study the annular invariant under crossing change and stabilizations in this complex. We give a slice-Bennequin like lower bound for band rank using these properties of this complex.

\begin{theorem}\label{theorem5}
Let $L$ be an annular link with $l$ components and $\mathcal{L}$ be the Legendrization of $L$. Then we have the following inequality,
\[ \mathscr{A}_{L}(t) \geq \frac{lk(U,L)t}{4}+ (1-\frac{t}{2})\frac{tb(\mathcal{L})+ |rot(\mathcal{L})|+l+1-lk(U,L)}{2}\]

holds for all $t \in [0,2]$.
\end{theorem}

Then, we turn our attention to braids. We find that the invariant is linear for quasi-positive braids.\\

\begin{prop} \label{generalquasi}
If $\beta$ is a quasi-positive braid of index $n$ with $l$ componenents. Then, $\mathscr{A}_{\beta}(t)=  t \frac{-wr(\beta)-l+n}{4}+\frac{wr(\beta)+l}{2}$.
\end{prop}

Braided cobordisms are the most interesting examples to investigate in this context. We specialize  the annular invariant for braids and braided cobordisms in which case it gives a lower bound on band rank. 
\begin{theorem}\label{theorem6}
Let $\beta$ be an $n$-braid with $l$ components and ${Id}_{n}$ be the identity $n$-braid. Then
$| \mathscr{A}_{{\beta}}(t)-\mathscr{A}_{{Id}_{n}}(t) | \leq \frac{rk_{n}(\beta)+l-n}{2} (1-\frac{t}{2})$.
\end{theorem}

Also, we get the following slice-Bennequin type lower bound on the band rank. A similar lower bound on four ball genus of a link was given by Cavallo \cite{cavallo}.

\begin{theorem}\label{theorem7}

If $\beta$ is a  $n$-braid  with $l$ componenets and $\mathcal{L}$ its Legendrization then,

\[  \frac{rk_{n}(\beta)+l-n}{2}  (1-\frac{t}{2}) \geq \frac{nt}{4}+ (1-\frac{t}{2})(\frac{tb(\mathcal{L})+ |rot(\mathcal{L})|+l+1-n}{2}) - \frac{n}{2} \] 

holds for all $t \in [0,2]$.
  
\end{theorem}

Finally, we explore some relations with transverse invariants and braid monodromy properties. In particular, we define subsets  $\mathscr{M}_{t}$  of braids such that for $0 \leq t_{1} \leq \cdots \leq t_{n} < 2$ , \[ QP \subseteq \mathscr{M}_{t_{1}} \subseteq \cdots \subseteq \mathscr{M}_{t_{n}} \subseteq RV . \] Where $QP$ and $RV$ denotes the monoids of Quasi-positive and right-veering braids respectively. \\

\begin{theorem}\label{theorem8} 

Membership in $\mathscr{M}_t$ is a transverse invariant and furthermore, $\mathscr{M}_t $ is a monoid. 
\end{theorem}

The paper is organized as follows. In Section 2, we review filtered combinatorial link Floer complex and their relevant properties. In Section 3, we define the annular invariant and prove its basic properties. In Section 4, we define a chain complex for annular links to investigate the refinement of $\hat{\theta}$ invariant. In Section 5, we define the reformulation in terms of t-modified complex and prove the key properties. Finally in Section 6, we study the invariant for braids and braieded cobordisms.\\

\date{\textbf{Acknowledgments:} I am grateful to my advisor, Olga Plamenevskaya, for pointing me towards annular Khovanov homology.

\pagebreak
\section{Filtered combinatorial link Floer complex}

In this section, we will introduce various flavors of grid homology [See \cite{Grid Homology for Knots and Links} for more details] and state the key results that will be used later.
\subsection{Grid states and gradings}
A grid state $x$ for a grid diagram $D$ with grid number $n$ consists of
$n$ points in the torus such that each horizontal and each vertical circle contains exactly one element of $x$. The set of grid states for $D$ is denoted by $S(D)$.\\

Given $x,y \in S(D)$, let $Rect(x,y)$ denote the space of embedded rectangles with the following properties. First of all, $Rect(x,y)$ is empty unless $x,y$ coincide at exactly $n-2$ points. An
element $r$ of $Rect(x,y)$ is an embedded disk in $\mathbb{T}$, whose boundary consists of four
arcs, each contained in horizontal or vertical circles; under the orientation induced
on the boundary of $r$, the horizontal arcs are oriented from a point in $x$ to a point
in $y$. The set of empty rectangles $r \in Rect(x,y)$ with $x \cap Int(r) = \phi$ is denoted
by $Rect^{o}(x, y)$. More generally, a path from $x$ to $y$ is a $1$-cycle $\gamma$ on $\mathbb{T}$ contained in the union of horizontal and vertical circles such that the boundary of
the intersection of $\gamma$ with the union of the horizontal curves is $y-x$ , and a domain
$\Delta$ from $x$ to $y$ is a two-chain in $\mathbb{T}$ whose boundary $\partial \Delta$ is a path from x to y. The set of domains from $x$ to $y$ is denoted $ \Pi(x, y)$.

Now let us define the gradings that will be used in various flavors of grid homology. $\mathbf{Maslov}$ $\mathbf{grading}$ function $M: S(D) \rightarrow \mathbb{Z}$ is defined as \[M(x) = \mathcal{J} (x - \mathbb{O},x - \mathbb{O})+1.\] Here for sets $P,Q$ of finitely many points in the grid, \[ \mathcal{J}(P,Q):=\sum\limits_{a \in P} \#\{ (a,b) \in (P,Q) | b \text{ has both coordinates strictly greater than  the ones of } a \} .\] For each component $L_{i}$ , the $\mathbf{Alexander}$ $\mathbf{grading}$ function $A_{i}: S(D) \rightarrow \frac{1}{2} \mathbb{Z}$  is defined as \[ A_{i}(x)=  \mathcal{J} (x - \frac{1}{2}(\mathbb{X}+\mathbb{O}),\mathbb{X}_{i} - \mathbb{O}_{i}) - \frac{n_{i}-1}{2}.\] 

If there is an empty rectangle $r$ between $x$ and $y$ and the Alexander and Masolov gradings  also satisfy, 

\[ M(y)-M(x)= -1 + 2\#(r\cap \mathbb{O}) \]

and,

\[A_{i}(y)-A_{i}(x)= \#(r\cap \mathbb{O}_{i}) - \#(r\cap \mathbb{X}_{i}) .\]\\

Total Alexander grading $A(x)$ is given by the sum of all component grading functions. It satisfies, 
\begin{equation}\label{algformula}
A(x)= \frac{1}{2}(\mathcal{J}(x - \mathbb{O},x - \mathbb{O}) - \mathcal{J}(x- \mathbb{X},x - \mathbb{X}) -\frac{n-1}{2}.
\end{equation}

There is also a winding number formula for computing Alexander gradings of link components. Let $w_{L_{i}}(q)$ denote the winding number of component $L_{i}$ around a point $q$ and $(P_{i})_{i=1,\cdots,8n}$ be the corners of X and O markings. Then,

\begin{equation}\label{windingformula}
A_{i}(x)=  -\sum_{p\in x} w_{L_{i}}(p) + \frac{1}{8} \sum_{j=1}^{8n} w_{L_{i}}(P_{i}) - \frac{1}{2}.
\end{equation} 
  
\subsection{Fully blocked filtered grid complex $\widetilde{\mathcal{GC}}$}

Given a toroidal grid diagram $D$ of a link $L$ , we associate to it a chain complex $(\widetilde{\mathcal{GC}}(D), \widetilde{ \partial})$
as follows.  $\widetilde{\mathcal{GC}}(D)$  is a free $\mathbb{F}_{2}$ module generated by grid states $S(D)$.\\

Given $x \in S(G)$, the differential map $\widetilde{\partial} : \widetilde{\mathcal{GC}}(D) \rightarrow \widetilde{\mathcal{GC}}(D)$, is defined in the following way ,  
\[ \widetilde{\partial}x := \mathlarger{\sum\limits_{y\in S(D)} \sum\limits_{r \in Rect^{o}(x,y), r \cap \mathbb{O} = \phi}}  y \ \ \forall x \in S(D) \]

It can be shown that $ \widetilde{ \partial} \circ  \widetilde{ \partial} = 0$  . Also, the Maslov grading when extended as a bi-linear form gives the homological grading of this complex. Moreover, Alexander grading functions $A_{i}(x)$ gives filtrations when extended as bi-linear form. Therefore,$(\widetilde{\mathcal{GC}}(D), \widetilde{ \partial})$ can be thought of as a filtered chain complex. Its homology, denoted by $\widetilde{\mathcal{GH}}(D) $  isomorphic to $\mathbb{F}_{2}^{2^{n-1}}$. the associated graded complex is denoted by $(\widetilde{GC}(D),\widetilde{\partial_{\mathbb{X}}})$. 

\subsection{Simply blocked filtered grid complex $\widehat{\mathcal{GC}}$ }

Let $D$ be the grid diagram of an oriented link $L$ with $l$ components. $\mathbb{O}$ is the set of O-markings in $D$. $ s \mathbb{O}     \subset \mathbb{O}$ is a subset that contains precisely one O marking from each component of $L$. Elements of $s\mathbb{O}$ are called special and represented as '$\phi$' in D. 
The simply blocked filtered chain complex $(\widehat{\mathcal{GC}}(D), \widehat{\partial)}$ is defined as  - \\
\[ \widehat{\mathcal{GC}}(D)= \text{Free }  \mathbb{F}_{2}[V_{1},V_{2},...,V_{n}] \ \text{ module over grid states } S(D) .\]

The differential $\partial$ is defined in the following way- \\ 
\[ \widehat{\partial}x := \mathlarger{\sum\limits_{y\in S(D)} \sum\limits_{r \in Rect^{o}(x,y), r \cap s\mathbb{O} = \phi}} V_{1}^{O_{1}(r)}...V_{m}^{O_{m}(r)} y \ \ \ \forall x \in S(D)\]

The homological grading of a generator $x \in S(D)$ in this complex is again given by Maslov grading. Multiplication by $V_{i}$ lowers the Maslov grading by $2$. $\widehat{\partial}$ lowers Maslov grading by $1$.\\

$(\widehat{\mathcal{GC}}(D),\widehat{\partial})$ comes with additional filtration(called Alexander filtration) for each link component.For a $x \in S(D)$ the $i$'th link component the Alexander filtration which are given by $A_i$s on generators. Then, $A_{i}$'s are extended to the whole module so that multiplication by $V_{k}$ lowers $A_{i}$ by $1$ if $X_{k} \in \mathbb{X}_{i}$ and it remains unchanged otherwise. \\

It can be shown that multiplication by each $V_{i}$ is filtered chain homotopic to $0$. So can be thought of as a $\mathbb{F}_{2}$ module. Furthermore its homology, denoted by $\widehat{\mathcal{GH}}(L)$, isomorphic to $\mathbb{F}_{2}^{2^{l-1}}$. Ozsvath and Szab\'{o} showed that the filtered chain homotopy type of $(\widehat{\mathcal{GC}}(D),\widehat{\partial})$ is an invariant of $L$.\\ 

$\tau(L)$ is particularly interesting invariant that one can extract from the filtered chain homotopy type. Let $G$ be grid diagram of link $L$, $\tau(L)$ is defined to be the minimal value of $i$ such that the inclusion $H_{0}(\mathcal{F}_{i}(\widehat{\mathcal{GC}}(G)) \rightarrow H_{0}(\widehat{\mathcal{GC}}(G))$ is non zero.

It turns out that $\tau(L)$ is a smooth concordance invariant. In fact, for a knot, $\tau$ gives lower bound on $4$-ball genus. It is also possible to define a $\tau$-set by considering other Maslov gradings.\\

Sometimes we will be interested in the associated graded complex , $(\widehat{GC}(D),\widehat{\partial_{\mathbb{X}}})$.Its homology groups are written as ${\widetilde{GH}}_{i}(L,j)$ where $i$ indicates Maslov grading and $j$ indicates Alexander grading. We also have, $ \widetilde{GH}(D) \cong \widehat{GH}(L)  \otimes W^{n-l}$. Notice that even though $\widetilde{GH}(D)$ depends on grid number,  $\widehat{GH}(L)$ is a link invariant independent of the diagram $D$.  $\widehat{GH}(L)$ is also referred to as combinatorial link Floer  homology or simply link Floer homology.

\subsection{Unblocked filtered grid complex $\mathcal{GC}^{-}$}

The unblocked grid chain complex $(\mathcal{GC}^{-}(D), \partial^{-})$ defined as  - \\
\[ {\mathcal{GC}}^{-}(D)= \text{Free }  \mathbb{F}_{2}[V_{1},V_{2},...,V_{n}] \ \text{ module over grid states } S(D) .\]

The differential $\partial^{-}$ is defined in the following way- \\ 
\[ {\partial}^{-}x := \mathlarger{\sum\limits_{y\in S(D)} \sum\limits_{r \in Rect^{o}(x,y)}} V_{1}^{O_{1}(r)}...V_{m}^{O_{m}(r)} y \ \ \ \forall x \in S(D) .\]

 The unblocked grid homology package also comes with the same extra filtrations (one for each link component). The filtered quasi-isomorphism type of the unblocked grid complex is an invariant of the link $L$. We will denote its associated graded complex by $(GC^{-}(D), \partial_{\mathbb{X}}^{-})$. The maps given by multiplication by $V_{i}$ and $V_{j}$ are chain homotopic in this complex if $O_{i}$ and $O_{j}$ are in the same link component. Therefore, ${GH}^{-}(D)$ can be thought of as a  $\mathbb{F}_{2}[V_{1},V_{2},...,V_{l}]$ module if the link $L$ represented by $D$ has $l$ components. If $L$ is a knot, it can be shown that  ${GH}^{-}(D)= F[U] \oplus Tor$ where $Tor$ is the torsion part. Further, the maximal Alexander grading of a non-torsion class is equal to $- \tau(L)$.

\subsection{Grid complexes of mirror links}\label{mirrorgh}
 
 Given a grid diagram $D$ (with grid number $n$) of a link $L$, let $D^{*}$ be the diagram obtained by reflecting $D$ through a horizontal axis. Then $D^{*}$ represents the link $m(L)$. Let $M$ and $M^{*}$ (resp. $A$ and $A^{*}$) denote Maslov (resp Alexander) grading in grids $D$ and $D^{*}$.    Let $x \rightarrow x^{*}$ be the natural bijection of the grid states induced by reflection. Then we observe that there is also a bijection between $Rect^{o}(x,y)$ and $Rect^{o}(y^{*},x^{*})$. So there is an isomorphism of chain complexes $\widetilde{\mathcal{GC}}(D^{*}) \cong Hom(\widetilde{\mathcal{GC}}(D), \mathbb{F}_{2}) $. Now, we can verify that $M(x)+ M^{*}(x^{*})=1-n$ and $A(x)+ A^{*}(x^{*})=l-n$ for a grid state $x$. Also, using the standard convention that the dual complex $\widetilde{GC}^{*} \cong Hom(\widetilde{\mathcal{GC}}(D), \mathbb{F}_{2})$ has grading and filtration level obtained by taking negative of those in $\widetilde{GC}$. Then we get an isomorphism of  $\widetilde{\mathcal{GC}}(D^{*}) \cong \widetilde{\mathcal{GC}}^{*}(D))[1-n,l-n]$. Now, we can pass to the hat version using Proposition \ref{collapse}, There, we have $\widehat{\mathcal{GC}}(D^{*}) \otimes W^{n-l}\cong \widehat{\mathcal{GC}}^{*}(D))\otimes{W^{*}}^{n-l}[1-l,0]$. So, we get an isomorphism $\widehat{\mathcal{GC}}(D^{*}) \cong \widehat{\mathcal{GC}}^{*}(D))[1-l,0]$. In fact, we can use the individual Alexander filtrations $\mathcal{F}_{i}$, $i=1,\cdots,l$  to put this isomorphism more generally as

 \[\widehat{\mathcal{GC}}(D^{*}) \cong \widehat{\mathcal{GC}}^{*}(D))[1-l,0,\cdots,0]. \]
 
 This fact will be used later.

\subsection{Invariants of Legendrian and transverse links}

The element $x^{+} \in S(D)$, which consists of the intersection points at the upper right corners of the squares containing the markings X in $D$, is a cycle in $({GC}^{-}(D), {\partial}_{\mathbb{X}}^{-})$. The element $x^{-} \in S(D)$ consisting of the intersection points at the south west corners of X markings is also a cycle. If $L$ is the Legendrian link corresponding to the grid diagram $ D$, then we know that $D$ represents the topological link type of $m(L)$. There is an interesting formula for the Alexander and Maslov gradings of the distinguished class in terms of the classical invariants of Legendrian link components. First, let us introduce some notations. Suppose $L_1,\cdots, L_{l}$ are the components of $L$. Define $tb_i(L)$ to be the linking number of $L_i$ with the Legendrian push-off $L'$. It can be checked that $tb(L)= tb_1(L)+\cdots+tb_{l}(L)$. $tb_i$  can be computed
in terms of a generic front projection $\textit{D}(L) = \cup_{1}^{l} \textit{D}(L_i)$ where $\textit{D}(L_i)$ is associated with $L_i$. \[tb_{i}(L)= wr(\textit{D}(L_i))+lk(\textit{D}(L_i),\textit{D}(L) \backslash \textit{D}(L_i)) - \frac{1}{2} \# \{ \text{cusps   in } \textit{D}(L_i) \}.\] 

Now define rotation numbers \[ r(L_{i}):=  \frac{1}{2} \# ( \{ \text{downward-oriented cusps in} \textit{D}(L_i) \} - \{ \text{upward-oriented cusps in} \textit{D}(L_i) \} ).\]
Again, it can be checked that $r(L)=r_{1}(L) + \cdots + r_{l}(L)$.

\begin{prop}\label{legcomplemma}

Let $x^+$ and $x^-$ be the distinguished cycles in the grids of $m(L)$ then,

\[M(x^+) = tb(L)- r(L) +1 \ \ \ \ M(x^-) = tb(L)+ r(L) +1.\]

\[A_i(x^+) = \frac{tb_{i}(L)- r_{i}(L) +1}{2} \ \ \ \ A_{i}(x^-) = \frac{tb_{i}(L)- r_{i}(L) +1}{2} .\]

\[A(x^+) = \frac{tb(L)- r(L) + l}{2} \ \ \ \ A_{i}(x^-) = \frac{tb (L)- r(L) + l}{2} .\]

\end{prop}
  
The above formulas will be used later in computations.\\

The homology classes $[x^{+}], [x^{-}]  \in GH^{-}(m(L))$, denoted by $\lambda^{+}(D)$ and $\lambda^{-}(D)$ respectively, are called the Legendrian grid invariant of $D$. For the transverse push-off $\mathcal{T}$ of an oriented Legendrian link $L$, the transverse grid invariant $\theta^{-}(D)$ is defined to be $\lambda^{+}(L) \in GH^{-}(m(L)) $. The following proposition states that the homological class is a well-defined invariant of Legendrian and transverse link types.

\begin{prop}\cite{theta} Let $D$ and $D'$ be two grid diagrams corresponding to Legendrian link $L$ (similarly transverse link $\mathcal{T}$), then there is an isomorphism \\

$\phi : GH^{-}(D) \rightarrow GH^{-}(D')$\\

such that $\phi(\lambda^{+}(D))=\lambda^{+}(D')$ and $\phi(\lambda^{-}(D))=\lambda^{-}(D')$  (similarly $\phi(\theta^{-}(D))=\theta^{-}(D')$). 
\end{prop}

Therefore, we choose to write $\lambda^{+}(D)$ as $\lambda^{+}(L)$ and $\lambda^{-}(D)$ as $\lambda^{-}(L)$  when $D$ corresponds to Legendrian link type $L$. Similarly, we will write $\theta^{-}(D)$ as $\theta^{-}(\mathcal{T})$ when $D$ corresponds to transverse link type $\mathcal{T}$. It is often more useful to consider the projection of $\theta(\mathcal{T})$ into $\widehat{GH}$, which we will call $\hat{\theta}(\mathcal{T})$. Projection of $\hat{\theta}(\mathcal{T})$. into $\widetilde{GH}(D)$ will be denoted as $\tilde{\theta}(\mathcal{T})$. It can be showed that $\hat{\theta}(\mathcal{T})=0$ if and only if  $\tilde{\theta}(\mathcal{T})=0$.\\

\subsection{Collapsed grid complexes }  

There are several other collapsed complexes that we can construct from the $\mathbb{F}_{2}[V_{1},V_{2},...,V_{n}]$  module $\mathcal{GC}^{-}(G)$ by setting some of the $V_{i}$'s equal to each other. The collapsed filtered grid complex is defined as $c\mathcal{GC}^{-}(G): = \frac{{\mathcal{GC}}^{-}}{V_{i_{1}}=V_{i_{2}}=...=V_{i_{l}}}$ , where $O_{i_{k}}$ is a $O$ marking belonging in the k'th link component. The associated graded version will be denoted by $cGC^{-}(G)$. Its homology, $cGH^{-}(G)$ can be thought of as a $\mathbb{F}_{2}[V]$ module. In fact, it can be shown that $cGH^{-}(G) \cong (\mathbb{F}_{2}[V])^{2^{l-1}} \oplus Tor $ (Here $Tor$ is the torsion part). As usual, $ cGH^{-}_{i}(G)$ denotes the homology at $i$'th Maslov grading.\\

We will also be interested in complexes where we collapse $V_{i}$ and $V_{j}$ with markings $O_{i}$ and $O_{j}$ belonging in the same component. 

\begin{prop}\label{collapse}
Let $G$ be a grid of link $L$. Suppose markings $O_{i}$ and $O_{j}$ belong to some link component $L_{N}$.
Then, there are filtered quasi-isomorphisms $\frac{\mathcal{GC}^{-}(G)}{V_{i}- V_{j}} \rightarrow \mathcal{GC}^{-}(G) \otimes W_{L_{N}}$ and $\frac{\widehat{\mathcal{GC}(G)}}{V_{i}- V_{j}} \rightarrow \widehat{\mathcal{GC}}(G) \otimes W_{L_{N}}$,  where $W_{L_{N}}$ is a $2$ dimensional graded vector space with one generator having (Maslov grading $=0$ , $A_{L_{N}}$ grading $=0$ , Alexander gradings for other link components  $=0$) and the other generator having   (Maslov grading $=-1$ ,  $A_{L_{N}}$ grading $=-1$,  Alexander gradings for other link components  $=0$). 
\end{prop}
\begin{proof}

Lets consider the short exact sequence \\
\begin{tikzcd}
 0 \arrow{r} & \mathcal{GC}^{-}(G) \arrow{r}{V_{i}- V_{j}}  & \mathcal{GC}^{-}(G)  \arrow{r}  & \frac{\mathcal{GC}^{-}(G)}{V_{i}- V_{j}} \arrow{r} & 0.
\end{tikzcd} \\

We know that the map given by multiplication by $V_{i}- V_{j}$ is chain homotopic to $0$. Also , multiplication by  $V_{i}- V_{j}$ lowers  Maslov grading by $2$, $A_{L_{N}}$ grading by $1$ and total Alexander grading by $1$. Therefore, mapping cone is filter-quasi-isomorphic to $\mathcal{GC}^{-}(G) \oplus \mathcal{GC}^{-}(G)[-1,-1,-1]$. So we have a quasi-isomorphism $\frac{\mathcal{GC}^{-}(G)}{V_{i}- V_{j}} \rightarrow \mathcal{GC}^{-}(G) \oplus \mathcal{GC}^{-}(G)[-1,-1,-1]$ and the conclusion follows. Similarly by setting on of the $U_{k}=0$ we can obtain the analogous result for the hat version.
\end{proof}

We can iterate the process in the above proposition as many times as we wish, but, we need to be careful when the complex has two markings in two different link components collapsed. It won't make sense to talk about link filtration for those particular components, and we need to talk about combined Alexander gradings.\\

\pagebreak
\section{Annular concordance invariant}
\subsection{Definition}
Let $L$ be an oriented link in $\mathbb{R}^{3}$. An annular link is $L'=L \cup U$ where $U$ is an unknot [See Fig ~\ref{anulink}].
We will assume that $U$ is oriented clockwise in x-y plane. So $-U$ will indicate the anticlockwise orientation. It should be mentioned that in our notation, for any oriented link $L$, $-L$ will denote $L$ with orientation reversed and $m(L)$ will denote mirror of $L$.\\

We will consider the grid chain complex $(\widehat{\mathcal{GC}}(D), \widehat{\partial)}$, where $D$ is a toroidal grid diagram of $U \cup L$ ;  $U$ being the unknot (oriented clockwise) and $L$ an oriented link. Suppose $D$ has grid number $n$ and $l$ is the number of components of $L$. We can write the set of non-special $O$- markings as $\mathbb{O} \setminus s\mathbb{O}=\{ O_{1},O_{2},..,O_{n-l} \}$. \\

We will call Alexander filtration for the unknot $U$ as $A_{U}$ and $(A_{1},...,A_{l})=$ Alexander filtrations for $l$ components of $L$. We denote the sum $A=A_{1}+..+A_{l}$ as simply $A_{L}$ and $A=A_{U}+A_{L}$ is the total Alexander grading.\\

\begin{figure}
\begin{center}
\includegraphics[width=0.30\textwidth]{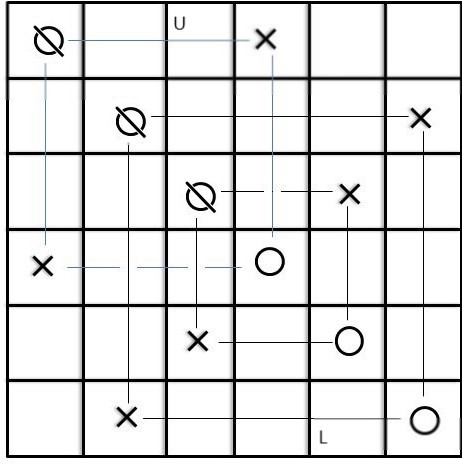}
\caption{Grid of $U \cup L$}\label{anulink}
\end{center}
\end{figure}

\begin{definition}
Define $\mathcal{F}_{t}(x)$ to be  $\frac{t}{2}A_{U}(x)+(1-\frac{t}{2})A_{L}(x)$ for each $0 \leq t \leq 2$ and $x \in \widehat{\mathcal{GC}}(D)$. 
\end{definition}

Since $\mathcal{F}_{t}(x)(\hat{\partial}(x))\leq\mathcal{F}_{t}(x)(x)$, $\mathcal{F}_{t}$  gives filtration levels defined as $\mathcal{F}_{t}^{s}(\widehat{\mathcal{GC}})=\{ a\in \widehat{\mathcal{GC}}(D) | \mathcal{F}_{t}(a) \leq s \} $ for every $s\in \mathbb{R}$. These filtration levels $ ..     \subseteq \mathcal{F}_{t}^{s} \subseteq \mathcal{F}_{t}^{s+1} \subseteq ..$ naturally induce filtration levels in the homology.  \\
 
Since $U \cup L$ has $l+1$ components from \cite{Grid Homology for Knots and Links}, we know that  $\widehat{\mathcal{GH}}(U\cup L)= H_{*}(\widehat{\mathcal{GC}},\widehat{\partial}) \cong (\mathbb{F}_{2})^{2^{l}}$ and $\widehat{\mathcal{GH}}_{0}(U\cup L) \cong \mathbb{F}_{2}$ ( $\widehat{\mathcal{GH}}_{0}$ is the homology at maslov grading $0$). 

 \theoremstyle{definition}
 \begin{definition}
 Define $\mathscr{A}_{L}^{U}(t):=min\{ s | H_{0}(\mathcal{F}_{t}^{s}(\widehat{\mathcal{GC}}),\widehat{\partial} ) \xrightarrow{i} \widehat{\mathcal{GH}}_{0}$ is nontrivial $\}$.\\
 \end{definition}
 
 Since $\mathcal{F}_{t}^{s}(\widehat{\mathcal{GC}})=\widehat{\mathcal{GC}}$ for sufficiently large $s$ and $\mathcal{F}_{t}^{s}(\widehat{\mathcal{GC}})$ is empty for sufficiently small $s$, it follows that $\mathscr{A}_{L}^{U}(t)$ is a finite real number.\\
 
Analogously, since the homology at Maslov grading $-l$ has rank equal to $1$, we can define 

 \begin{definition}
 Define $\mathsf{A}_{L}^{U}(t):=min\{ s|H_{-l}(\mathcal{F}_{t}^{s}(\widehat{\mathcal{GC}}),\widehat{\partial} ) \xrightarrow{i} \widehat{\mathcal{GH}}_{-l}$ is nontrivial $\}$.\\
 \end{definition}

We will drop $U$ for making the notation look less cumbersome (with a slight abuse of notation) and write $\mathscr{A}_{L}^{U}(t)$ and  $\mathsf{A}_{L}^{U}(t)$ as $\mathscr{A}_{L}(t)$ and $\mathsf{A}_{L}(t)$.
 
 Since filtered quasi-isomorphism type of the complex is an invariant of the link $U\cup L$, it follows that $\mathscr{A}_{L}(t)$ and $\mathsf{A}_{L}(t)$ are annular link invariants. \\

\subsection{Adding special O-markings}

In this section, we will show that if we add any number of special O markings the invariant (defined in the same way in that diagram) remains the same.

In particular, we may define the invariant in the $\widetilde{\mathcal{GC}}$ version which is useful for computations.
 \theoremstyle{definition}
 \begin{definition}
 $\widetilde{\mathscr{A}_{L}(t)}:=min\{ s|H_{0}(\mathcal{F}_{t}^{s}(\widetilde{\mathcal{GC}}),\widetilde{\partial} ) \xrightarrow{i} \widetilde{\mathcal{GH}}_{0}$ is nontrivial $\}$.
 \end{definition}
  We have the following equivalence, 
\begin{prop}\label{impvecdef}
$\mathscr{A}_{L}(t)= \widetilde{\mathscr{A}_{L}(t)}$. 
\end{prop}
\begin{proof}
Let $W=$ 2 dimensional graded vector space generated by  $ \{ v_{L}^{+},v_{L}^{-} \}$ ( with $(M,A_{U},A_{L})(v_{L}^{+})=(0,0,0)$ and $(M,A_{U},A_{L})(v_{L}^{-})=(-1,0,-1)$ )\\

and, $W_{U}=$ 2 dimensional graded vector space generated by  $ \{ v_{U}^{+},v_{U}^{-} \}$ ( with $(M,A_{U},A_{L})(v_{U}^{+})=(0,0,0)$ and $(M,A_{U},A_{L})(v_{U}^{-})=(-1,-1,0)$). \\

We know that   $(\widetilde{\mathcal{GC}},\widetilde{\partial})$ is ($A_{L}$,$A_{U}$) filtered quasi-isomorphic to $(\widehat{\mathcal{GC}}\otimes W^{\otimes n-l-2} \otimes W_{U},\widehat{\partial})$[From Proposition \ref{collapse}.  Since $0$ is the highest Maslov grading in $\widehat{\mathcal{GH}}$, it follows that $\mathscr{A}_{L}(t)= \widetilde{\mathscr{A}_{L}(t)}$.
\end{proof}

Similarly, we can define

 \theoremstyle{definition}
 \begin{definition}
 $\widetilde{\mathsf{A}_{L}}(t):=min\{ s|H_{1-n}(\mathcal{F}_{t}^{s}(\widetilde{\mathcal{GC}}),\widetilde{\partial} ) \xrightarrow{i} \widetilde{\mathcal{GH}}_{1-n}$ is nontrivial $\}$ [ where $n$ is the grid number of $U\cup L$].
 \end{definition}
 
 Then, we have the following relation
\begin{prop}
$\widetilde{\mathsf{A}}_{L}(t)= \mathsf{A}_{L}(t)  - \frac{t}{2} - (n-l-2)(1-\frac{t}{2})$. 
\end{prop}
\begin{proof}

Again, since homology at Maslov grading $1-n$ has rank $1$, we have   $(\widehat{\mathcal{GC}},\widetilde{\partial})$ is ($A_{L}$,$A_{U}$) filtered quasi-isomorphic to $(\widehat{\mathcal{GC}}\otimes W^{\otimes n-l} \otimes W_{U},\widehat{\partial})$. Now, $\mathcal{F}^{t}( \widetilde{\mathcal{GH}}(U\cup L)_{1-n}=\mathcal{F}^{t}( \widetilde{\mathcal{GH}}(U\cup L)_{l}\otimes v_{U}^{-} \otimes{ v_{L}^{-}}^{\otimes (n-l-2)}) $. Therefore, it follows from the definition that $\widetilde{\mathsf{A}}_{L}(t)= \mathsf{A}_{L}(t) - \frac{t}{2} - (n-l-2)(1-\frac{t}{2})$.
\end{proof}

\subsection{Properties of the invariant}

\begin{prop}
$\mathscr{A}_{-L}^{-U}(t) =  \mathscr{A}_{L}^{U}(t) $ and $\mathcal{A}_{L}^{U}(t)=\mathcal{A}_{-L}^{-U}(t)$ . 
\end{prop}
\begin{proof}

As in [3], we can consider a map $\Phi$ from the diagram of $U\cup L$ to the reflection of the diagram along a diagonal which is a diagram of $-(U\cup L)$. We have $A_U(\Phi(x))=A_U(x)$ and $A_L(\Phi(x))=A_L(x)$.  Then $\Phi$ is a filtered isomorphism. Therefore, it follows that if both orientations are reversed, the invariants don't change.
\end{proof}

\begin{prop}\label{disjointsum}
$\mathscr{A}_{L_{1}\sqcup L_{2}}(t)= \mathscr{A}_{L_{1}}(t)+ \mathscr{A}_{L_{2}}(t)$. 
\end{prop}
\begin{proof}
We observe that $(L_{1} \sqcup L_{2}) \cup U$ can be viewed as $(L_{1} \cup U_{1}) \# (L_{2} \cup U_{2}) $ where the connected sum is obtained as $(U_{1} \# U_{2}) \sqcup  (L_{1} \sqcup L_{2})$. Let $A_{U}$ be equal to $A_{ U_{1}} + A_{ U_{2}}$ . From the Kunneth formula ( See  Theorem 11.1 in \cite{HolomorphicLinkInvariants} ), we know that $\widehat{\mathcal{GC}}((L_{1} \cup U_{1}) \# (L_{2} \cup U_{2}))$ is $(A_{U},A_{L_{1}},A_{L_{2}})$-filtered isomorphic to  $\widehat{\mathcal{GC}}(L_{1} \cup U_{1}) \otimes \widehat{\mathcal{GC}}((L_{2} \cup U_{2}) $. Therefore, they are also $\mathscr{F}_{t}$ filtered isomorphic and the conclusion immediately follows. 
\end{proof}

\begin{prop}

$\mathscr{A}_{L}(t)= - \mathcal{A}_{m(L)}(t)$.
\end{prop}

\begin{proof}

Let $D$ be the grid diagram for $L \cup U$, and let $D^{*}$ be the diagram obtained by reflecting  $D$ through a horizontal axis. Then  $D^{*}$ represents  $m(L \cup U)$.\\
We know that $\widehat{GC}(D^{*})$ is filtered isomorphic to $\widehat{GC}^{*}(D)[l-1,0,\cdots,0]$ [See Section \ref{mirrorgh}].

This implies $H_{*}(\mathcal{F}_{t}^{s}(\widehat{GC}(D^{*})))= H_{* - (1-l)}({{\mathcal{F}}^{*}}_{t}^{s}({\widehat{GC}}^{*}(D)))$. Therefore, the conclusion follows.  

\end{proof}

We can extract two invariant functions from the invariant $\mathscr{A}_{L}(t)$. For any $t_{0} \in[0,2)$, the slope function

\[ m_{t_{0}}(L):= \lim_{t\to {t_{0}^{+}}}\frac{A_{L}(t)-A_{L}(t_{0})}{t-t_{0}} .\]
We will also assume $m_{2}(L)=0$. We also define the y-value function $y_{t_{0}}(L) := \mathscr{A}_{L}(t_{0})-t_{0}m_{t_{0}}(L)$.

\begin{prop}\label{slopeprop}
(i)$\mathscr{A}_{L}(t)$ is a continuous piece-wise linear function.\\
(ii) At a non-singular point $t_{0}$, the slope $m_{t_{0}}$ is equal to $\frac{A_{U}(x_{0})-A_{L}(x_{0})}{2}$ for some generator $x_{0} \in \widehat{\mathcal{GC}}(U \cup L)$.  \\
(iii)If $t_{0}$ is a singular point, then the absolute value of change in slope $|\Delta m_{t_{o}}|$ is equal to $|\frac{A_{L}(x_{2})-A_{L}(x_{1})}{t_{0}}|$ for some generators  $x_{1},x_{2} \in \widehat{\mathcal{GC}}(U \cup L)$.
\end{prop}

\begin{proof}

There are only finitely many elements $x_{i} \in \widehat{\mathcal{GC}}(U \cup L)$  that are generators of homology at Maslov degree $0$. Lets consider all linear functions $G_{x_{i}}(t)= \frac{t}{2}A_{U}(x_{i}) + (1-\frac{t}{2})A_{L}(x_{i})$. Then, $\mathscr{A}_{L}(t)= \min_{i} G_{x_{i}}(t)$. It follows that $\mathscr{A}_{L}(t)$ is a continuous piece-wise linear.\\

At each non-singular point $t_{0}$, there must be some generator $x_{0}$ such that $\mathscr{A}_{L}(t)= \ G_{x_{0}}(t)$ $\forall t\in (t_{0}- \delta, t_{0} + \delta)$ for some $\delta > 0$. Hence, the slope of $G_{x_{0}}(t)$ is equal to $m_{t_{0}}$. \\

At a singular point $t_{0}$ assume there is a generator generator $x_{1}$ that assumes the value of the invariant for $t$'s slightly less than $t_{0}$ and $x_{2}$ assumes the value for $t$'s slightly greater than $t_{0}$. Then at $t_{0}$, we must have $\frac{t_{0}}{2} A_{U}(x_{1})+(1- \frac{t_{0}}{2} ) A_{L}(x_{1})= \frac{t_{0}}{2} A_{U}(x_{2})+(1- \frac{t_{0}}{2} ) A_{L}(x_{2}) $ . So, $t_{0}\Delta m_{t_{0}} = A_{L}(x_{2})-A_{L}(x_{1})$ and the conclusion follows.  
\end{proof}

Let $\tau$ be the Cavallo's invariant \cite{cavallo} for links. 
\begin{prop}

$\mathscr{A}_{L}(\frac{1}{2})=\frac{\tau(U\cup L)}{2}$.
\end{prop}

\begin{proof}

We have, $\mathcal{F}_{\frac{1}{2}}= \frac{1}{2}A_{U}+\frac{1}{2}A_{L}=\frac{1}{2}A$.
Therefore,
$\mathscr{A}_{L}^{U}(\frac{1}{2}):=min\{ s | H_{0}(\mathcal{F}_{\frac{1}{2}}^{s}(\widehat{\mathcal{GC}}),\widehat{\partial} ) \xrightarrow{i} \widehat{\mathcal{GH}}_{0}$ is nontrivial $\}=\frac{1}{2} \tau(U\cup L)$.
\end{proof}

\subsection{Some computations }

We will give two sample computations for torus braids and trivial braid. We will derive a more general formula for quasi-positive braids in Proposition \ref{generalquasi}.

\begin{figure}[h!]
\begin{center}
\includegraphics[width=0.3\textwidth]{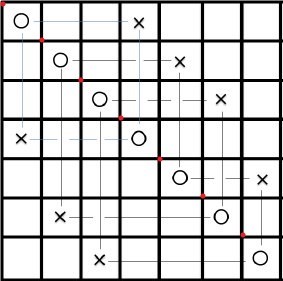}
\end{center}
\caption{Torus braid} \label{torus}
\end{figure}

\begin{prop}
$\mathscr{A}_{T_{p,q}}(t)=\frac{pq-q+l}{2}+ \frac{t}{4} (p+q-pq-l)$ where $T_{p,q}$ is the torus braid with $l$ components and $0 \leq p \leq q $. 
\end{prop}
\begin{proof}
In case of annular torus links $T_{p,q}$ with $p \leq q$, the grid diagram( of  $T_{p,q} \cup U$) [See Fig ~\ref{torus}] contains a unique generator  $x_{NWO}$ \cite{sarkar} of the filtered complex at Maslov grading $0$. We can compute the  $A_{U}$ and $A_{L}$ gradings of the generator $x_{NWO}$ representing the cycle to determine the invariant. We use the winding number formula [Equation \ref{windingformula}] and Alexander grading formula [Equation \ref{algformula}] to derive,
\[A_{U}(x_{NWO})= \frac{p}{2} \]

and 
\[ A_{T_{p,q}}(x_{NWO})=\frac{pq-q+l}{2} .\]
Therefore,
\[\mathscr{A}_{T_{p,q}}(t)=\frac{pq-q+l}{2}+ \frac{t}{4} (p+q-pq-l)\]

\end{proof}

\begin{prop}
$\mathscr{A}_{I_{n}}(t) =  \frac{n}{2} $ where $I_{n}$ is the trivial braid with $n$ strands.
\end{prop}
\begin{proof}
As in the case of annular torus links, the grid diagram contains a unique cycle in Maslov grading $0$ represented by the generator $x_{NWO}$. Also,
\[A_{U}(x_{NWO})= \frac{n}{2} \]
and
\[A_{I_{n}}(x_{NWO})=\frac{n}{2}.\]
Therefore,
\[\mathscr{A}_{I_{n}}(t)= \frac{n}{2}\]

\end{proof}

\subsection{Annular concordance}

\theoremstyle{definition}
\begin{definition}
An  annular cobordism $\Sigma$ between two annular links $L_{1} \in S^{3} \times \{0\}$ and $L_{2} \in S^{3} \times \{1\}$ is an embedded surface in $S^{3} \times [0,1]$  which is disjoint from z-axis in each $ S^{3} \times \{i\}$ for $i \in [0,1]$ and satisfying $\partial \Sigma= L_{1} \sqcup - L_{2}$.
\end{definition}

We say two links $L_1$ and $L_2$ are annular concordant if there is an annular cobordism of genus $0$ between $L_1$ and $L_2$. Any annular cobordism can be represented by a sequence of an identity, annular split, annular merge, annular birth, and annular death cobordisms. An annular cobordism $\Sigma$ between two links $L_{1}$ and $L_{2}$ is called $\mathbf{strong}$ if the connected components of $\Sigma$ are knot cobordisms between a two components of $L_{1}$ and $L_{2}$. Any strong annular cobordism can be perturbed so that it is a composition of torus cobordisms, annular birth followed by merge cobordisms and split followed by annular death cobordisms.\\ 
 \\

\begin{figure}[h!]
\begin{center}
\includegraphics[width=0.5\textwidth]{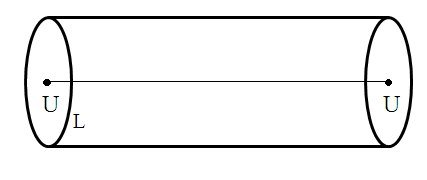}
\end{center}
\caption{Identity cobordism}
\end{figure}
 By taking slices of a cobordism, any cobordism can also be seen as a movie in the link diagram. An identity cobordism is a cobordism between a link the movie is represented by the three Reidemeister moves performed. For birth and death, the movie corresponds introduction or deletion of an unknotted circle. For merge or split cobordisms the change is represented by perturbing the link diagram like Figure \ref{sdl}. Similarly, we can easily figure out annular version of the movie [See \cite{annular} for more details]. \\ 
\begin{figure}[h!] 
    \begin{center}    
    \includegraphics[width=.3\textwidth]{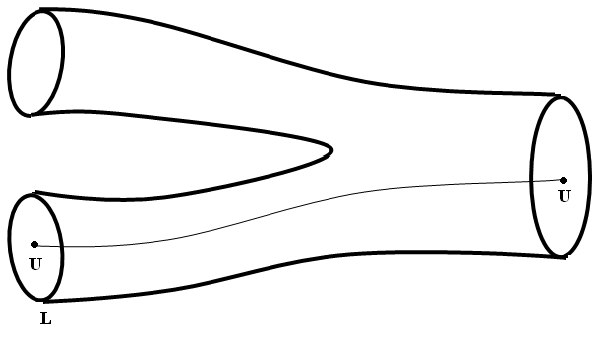}
    \end{center}
   \caption{Annular merge cobordism}
 \end{figure}

We know that each of these moves induces filtered maps of some degree. We check that these maps are both $A_{U}$ and $A_{L}$ filtered. We also compute the $\mathcal{F}_{t}$ grading shift that will give bounds on cobordism genus. Our construction follows the prescription given in \cite{sarkar}.\\

\begin{figure}[h!] 
    \begin{center}    
    \includegraphics[width=.3\textwidth]{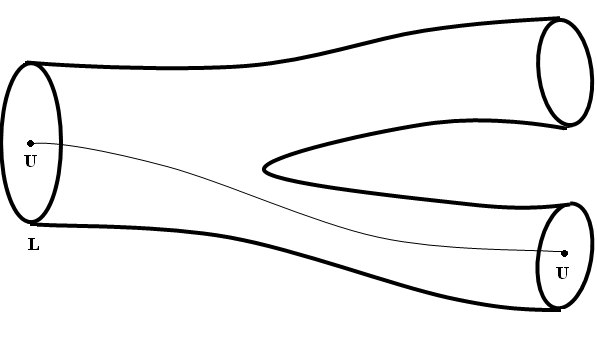}
    \end{center}
   \caption{Annular Split cobordism}
 \end{figure}

\begin{enumerate}
\item \textbf{Identity:} These maps are filtered quasi-isomorphisms, hence $\mathcal{F}_{t}$ filtered of degree 0.
\item \textbf{Split:} If the grid diagram $D_{2}$ of $L_{2}$ is obtained from $D_{1}$ of $L_{1}$ by a split move, then it can represented as the effect of swapping the positions of two Xs(in one of the components of $L_{1}$ in $2*2$ block. Notice that we need an extra special marking  in $D_{2}$ for $\widehat{\mathcal{GC}}$ version. So, we consider the $\widetilde{\mathcal{GC}}$ version as done in \cite{sarkar}.  \\

\begin{figure}[h!] 
\begin{center}
\includegraphics[width=.4\textwidth]{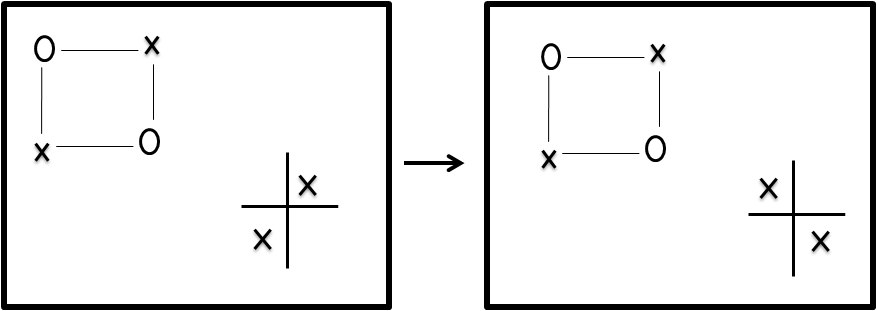}
\end{center}    
   \caption{Grid move corresponding to Merge and Split cobordisms}\label{sdl}
 \end{figure}
 
Consider $ Id : \widetilde{\mathcal{GC}}(D_{1}) \rightarrow \widetilde{\mathcal{GC}}(D_{2})$. It induces an isomorphism in homology as it is unaffected by O swaps. So we just need to compute $A_{U}$ and $A_{L}$ degree shifts. \\

We have the following formula (See \cite{Grid Homology for Knots and Links}) for alexander filtration for unknot in terms of winding numbers-  $A_{U}(\mathbf{x}) = \sum \limits_{x \in \mathbf{x}} w_{U}(x) + \frac{1}{8} \sum \limits_{j=1}^{8n}w_{U}(p_{j}) - 1 $ where $p_{j}$ are corners of Xs and Os. Clearly swapping two Xs in one of the components of $L$ doesn't affect any of the winding numbers. Therefore, $A_{U}$ degree shift is $0$.  \\
Now $A_{L_{1}}(x)-A_{L_{2}}(x)=\frac{1}{2}(\mathcal{J} (x -\mathbb{X}_{L_{1}},x - \mathbb{X}_{L_{1}})-\mathcal{J} (x -\mathbb{X}_{L_{2}},x - \mathbb{X}_{L_{2}})) + \frac{1}{2}=\frac{1}{2}(\mathcal{J}(x,\mathbb{X}_{L_{2}})-\mathcal{J}(x,\mathbb{X}_{L_{1}})+\mathcal{J}(\mathbb{X}_{L_{2}},x)-\mathcal{J}(\mathbb{X}_{L_{1}},x)+\mathcal{J}(\mathbb{X}_{L_{1}},\mathbb{X}_{L_{1}})-\mathcal{J}(\mathbb{X}_{L_{2}},\mathbb{X}_{L_{2}}))+ \frac{1}{2}=1$[Since the quantity $\mathcal{J}(\mathbb{X}_{L_{1}},\mathbb{X}_{L_{1}})-\mathcal{J}(\mathbb{X}_{L_{2}},\mathbb{X}_{L_{2}})=1$  from the diagram]. Here $\mathbb{X}_{L_{1}}$ are the X's in the grid diagram of $L_{1}$ and  $\mathbb{X}_{L_{2}}$ are X's in grid diagram of $L_{2}$.\\

Passing to the $\widehat{\mathcal{GC}}$ version, we get a quasi-isomorphism $\Phi_{merge} : \widehat{\mathcal{GC}}(L_{1}) \rightarrow \widehat{\mathcal{GC}}(L_{2})\otimes W$ of $\mathcal{F}_{t}$ filtered of degree $1-\frac{t}{2}$. \\

\item \textbf{Merge:}
We use the same construction for the merge move. $A_{U}$ degree shift is $0$ by the same argument, but for the $A_{L}$ grading, since number of link component is decreasing by $1$ here, $A_{L_{1}}(x)-A_{L_{2}}(x)=\frac{1}{2}(\mathcal{J} (x -\mathbb{X}_{L_{1}},x - \mathbb{X}_{L_{1}})-\mathcal{J} (x -\mathbb{X}_{L_{2}},x - \mathbb{X}_{L_{2}})) - \frac{1}{2}=0$.\\

Again by passing to the $\widehat{\mathcal{GC}}$ version, we get a quasi-isomorphism $\Phi_{split} : \widehat{\mathcal{GC}}(L_{1}) \otimes W \rightarrow \widehat{\mathcal{GC}}(L_{2})$ of $\mathcal{F}_{t}$ filtered of degree $0$.

\begin{figure}[h!]
\begin{center}

    \includegraphics[width=.4\textwidth]{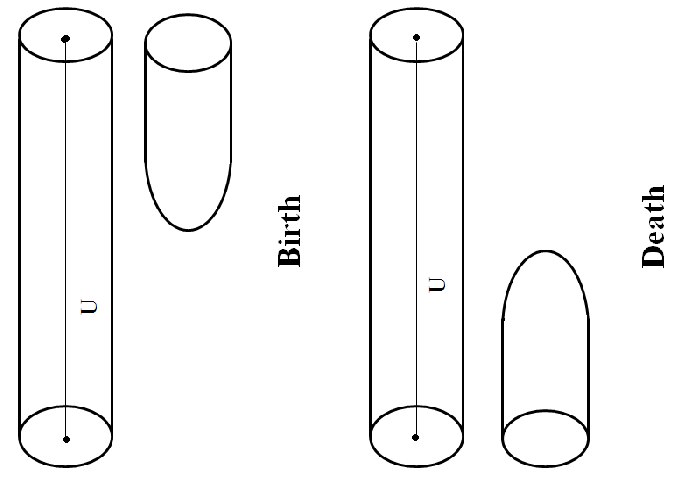}
\end{center}    
   \caption{Annular Birth and death cobordisms}
 \end{figure}

\item \textbf{Birth:}
If $L_{2}$ (with grid $D_{2}$ is obtained from $L_{1}$ (with grid $D_{1}$) by birth, we know from \cite{sarkar} that there is a quasi-isomorphism from $\widehat{\mathcal{GC}}(L_{1})$ to $\widehat{\mathcal{GC}}(L_{2})$ given by \\

$ s(x)=\sum\limits_{y\in S(D_{1})}\sum\limits_{H \in s\mathscr{L}(i(x),y,x), H \cap s\mathbb{O}=\phi} V_{1}^{n_{1}(H)}...V_{m}^{n_{m}(H)} y $ for any $x \in S(D_{1})$. Here  $s\mathscr{L}(i(x),y,x)$ are snail like domains centered at $c$ joining $i(x)$ to $y$ and $n_{i}(H)$ is the number of times its passes through $O_{i}$.\\

\begin{figure}[h!]
\includegraphics[width=1 \textwidth]{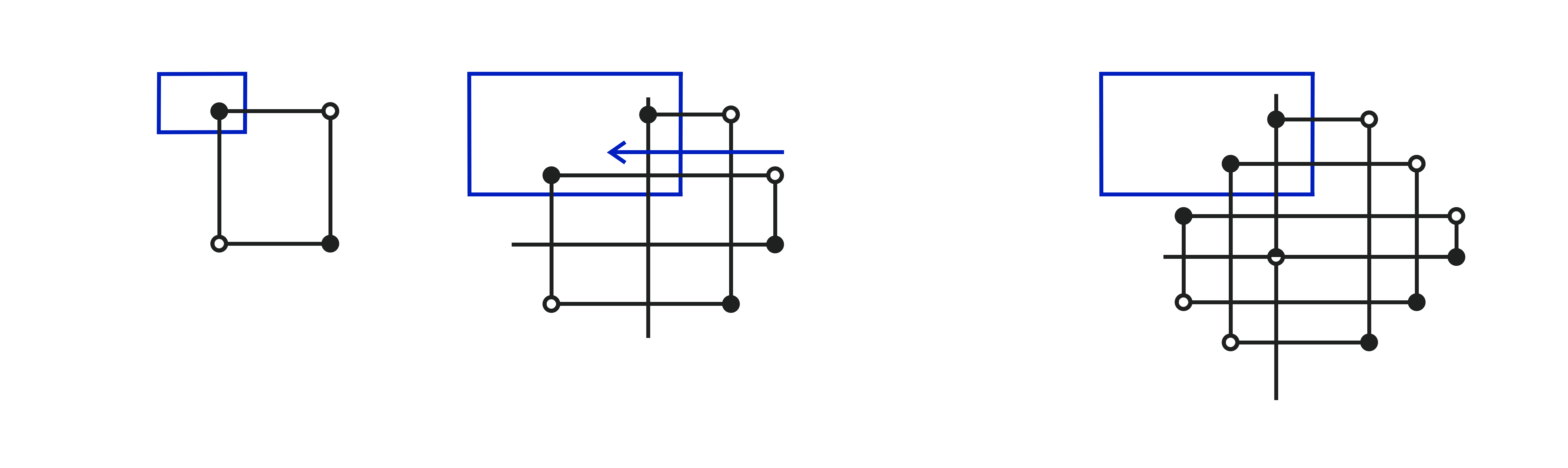}
\caption{Snail like domains in the birth move}
\end{figure}
    
Now let us assume $O_{1}$ be the non-special O marking in $U$. Now, for any $H \in s\mathscr{L}(i(x),y,x)$ we have $A_{U}(x)-A_{U}(y)=-\sum($ unknot winding numbers for $x$ points $) + \sum ($ unknot winding numbers for $y$ points $)=n_{1}(H)$.Therefore $A_{U}(x)=A_{U}( V_{1}^{n_{1}(H)}...V_{m}^{n_{m}(H)} y)$. So $s$ is $A_{U}$ filtered.\\

We already know that $s$ is $A$ filtered from [2]. Therefore, $s$ is $\mathcal{F}_{t}$ filtered of degree $0$.

\item \textbf{Death:} If we compose annular birth cobordism with a merge then we get a cobordism $\mathbf{BM}$ which induces a quasi-isomorphism $\Phi_{BM} : \widehat{\mathcal{GC}}(L_{1}) \rightarrow \widehat{\mathcal{GC}}(L_{2})$  which is  $\mathcal{F}_{t}$ filtered of degree $0$. It can also be seen as $\Phi_{BM}^{*} : \widehat{\mathcal{GC}}(L_{2}^{*}) \rightarrow \widehat{\mathcal{GC}}(L_{1}^{*})$. It follows that  $\Phi_{BM}^{*}$ is also  $\mathcal{F}_{t}$ filtered of degree $0$. Now we observe that the cobordism from $L_{2}^{*}$ to $L_{1}^{*}$ is a split move followed by annular death,i.e. $\Phi_{BM}^{*}= \Phi_{Split} \circ \Phi_{Death}$. Therefore, annular death induces the quasi-isomorphism $ \Phi_{Death}$ which is $\mathcal{F}_{t}$ filtered of degree $-1+\frac{t}{2}$.

\end{enumerate}

We will call annular merge cobordism followed by an annular split cobordism(or annular split cobordism followed by an annular merge cobordism) a $\mathbf{torus}$ $\mathbf{cobordism}$. Clearly torus cobordism induces a  quasi-isomorphism $\Phi_{T}:  \widehat{\mathcal{GC}}(L_{1}) \rightarrow \widehat{\mathcal{GC}}(L_{2}) $ of filtered degree $1-\frac{t}{2}$.\\

\begin{proof} [Proof of Theorem \ref{theorem1}]
Since any strong cobordism of genus $g$ can be written as composition of $g$ torus cobordisms, some annular birth followed by merge and some split followed by annular death cobordisms. Both annular birth followed by merge and some split followed by annular death are filtered of degree $0$. So, we get a quasi-isomorphism $\Phi : \widehat{\mathcal{GC}}(L_{1}) \rightarrow \widehat{\mathcal{GC}}(L_{2})$ with filtered degree $g(1-\frac{t}{2})$. So we have a commutative diagram \\

\[\begin{tikzcd}[
  ar symbol/.style = {draw=none,"\textstyle#1" description,sloped},
  isomorphic/.style = {ar symbol={\cong}},
  ]
H_{0}(\mathcal{F}_{t}^{s}(\widehat{\mathcal{GC}}(U \cup L_{1})) \arrow{r}{\Phi} \arrow[swap]{d}{i} &H_{0}(\mathcal{F}_{t}^{s+g(1-\frac{t}{2})}(\widehat{\mathcal{GC}}(U \cup L_{2})) \arrow{d}{i} \\
\widehat{\mathcal{GH}}_{0}(U \cup L_{1}) \arrow{r}{Id} & \widehat{\mathcal{GH}}_{0}(U \cup L_{2}) \\
  \mathbb{F}_{2} \ar[u,isomorphic] & \mathbb{F}_{2} \ar[u,isomorphic]
\end{tikzcd}
.\]

It follows that $\mathscr{A}_{L_{2}}(t) \leq \mathscr{A}_{L_{1}}(t)+ g(1-\frac{t}{2})$. Similarly by looking at the cobordism from $L_{2}$ to $L_{1}$, we can show that  $\mathscr{A}_{L_{1}}(t) \leq \mathscr{A}_{L_{2}}(t)+ g(1-\frac{t}{2})$.

\end{proof}

\begin{cor}
If $L_{1}$ and $L_{2}$ are annular concordant then $\mathscr{A}_{L_{1}}(t)=\mathscr{A}_{L_{2}}(t)$. 
\end{cor}

\pagebreak

\section{Refinement of $\hat{\theta}$ invariant} \label{refinesection}

\subsection{Braid grid complex}

\begin{figure}
\begin{center}
\includegraphics[scale=0.6]{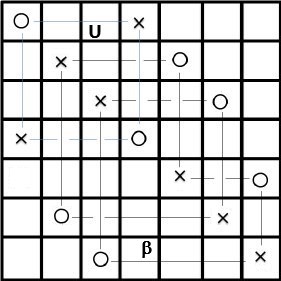}
\end{center}

\caption{Diagram of  $U \cup m(\beta)$ } \label{k1}
\end{figure}
Let $\beta$ be a $N$-braid. We can consider the the grid diagram $D$ of $U \cup m(\beta)$ [See Fig. \ref{k1}] where the unknot $U$ (oriented clockwise) acts as a braid axis. Also we assume that the unknot is linked negatively with the braid. Let $\mathbb{X}=\{X_{1},X_{2},X_{3},\cdots,X_{n} \}$  and $\mathbb{O}=\{O_{1},O_{2},O_{3},\cdots,O_{n} \}$ be the sets of X markings and O markings respectively where $X_{1}, X_{2}, O_{1} \text{ and } O_{2}$ represent the markings of the unknot $U$. We will use the notation $O_{\beta}(r)$ to denote the number of O markings belonging to the $\beta$ component inside a rectangle $r$.

\begin{definition}
 Define the chain complex $({\mathcal{C}}_{U \cup \beta}(D),\partial)$ as a $\mathbb{F}_{2}[V_{3},V_{4},\cdots,V_{n}]$-module over grid states  $S(D)$ 
 \[ {\partial}x := \mathlarger{\sum\limits_{y\in S(D)} \sum\limits_{r \in Rect^{o}(x,y), r \cap \mathbb{X} = \phi}} V_{3}^{O_{3}(r)}V_{4}^{O_{4}(r)}\cdots V_{n}^{O_{n}(r)} y \ \ \ \forall x \in S(D) \]
\end{definition}

From now on we will refer to a $\mathbb{F}_{2}[V_{3},V_{4},\cdots,V_{n}]$-module as $\mathcal{R}$-module.

We review the construction used in \cite{lossbraid} and show how their chain complex is related to ours. A multi-pointed Heegaard diagram for an oriented link $L \subset Y$ is given by a ordered tuple $H =(\Sigma,\alpha,\beta, z_L,w_L \cup w_f)$, where $w_f$ is the set of free base points. A grid diagram  can be naturally viewed as a multi-pointed Heegard diagram embedded in torus. Lets denote the  the grid of $\beta \cup U$ as the multi-pointed Heegard diagram, $\mathcal{H}_{1}=(T^{2},\alpha,\beta,z_{\beta}\cup z_{U},w_{\beta}\cup w_{U})$. Here $z$s denote the X markings and $w$s denote the O markings. If we drop the points $z_{U}$, then we get the variant  $\mathcal{H}_{4}=(T^{2},\alpha,\beta,z_{\beta},w_{\beta}\cup w_{U})$ where the set $w_{U}$ work as free basepoints. This variant is considered in \cite{lossbraid} to introduce the reformulation of GRID invariant. $CFK^{-,2}({\mathcal{H}}_{4})$ denotes the knot Floer complex associated with  $\mathcal{H}_{4}$.   

\begin{prop}\label{impiso}
$\mathcal{C}_{U\cup \beta}$ is filtered quasi-isomorphic to $CFK^{-,2}({\mathcal{H}}_{4})$.
\end{prop}
\begin{proof}

It is easy to see that these complexes are isomorphic after reversal of roles of X and O markings in the unknot component, which also explains the orientation convention of the unknot in these two complexes.

\end{proof}

In their paper \cite{lossbraid} , they also show that $H_{top}( {\mathcal{F}_{U}}^{bot}(\mathcal{C}_{U \cup \beta}))$ is generated by the distinguished cycle $[x_{4}]$ which in our complex is the state consting of north-east corners of X-markings (Here $bot$ is minimum value of  $\mathcal{F}_{U}$). We will use this relationship to refine the transverse invariant.

\begin{prop}
Multiplication by $V_{i}$ is chain homotopic to multiplication by $V_{j}$ if $O_{i}$ and $O_{j}$ belong to the same link component in $\beta$.
\end{prop}
\begin{proof}
Let $X_{k}$ be the X-marking that is in same row as $O_{m}$ and in the same column as $O_{n}$.
Define $H_{X_{k}}:\mathcal{C}_{U\cup \beta} \rightarrow \mathcal{C}_{U\cup \beta}$,
 \[ H_{X_{k}}(x) := \mathlarger{\sum\limits_{y\in S(D)} \sum\limits_{r \in Rect^{o}(x,y), r \cap \mathbb{X} = {X_{k}}} } V_{3}^{O_{3}(r)}V_{4}^{O_{4}(r)}\cdots V_{n}^{O_{n}(r)} y \ \ \ \forall x \in S(D) \]

 Then, in the compositions of rectangles appearing in $\partial H_{X_{k}} + H_{X_{k}} \partial$, contributions from all but two annuli containing $X_{k}$ cancel. So we get,

\[ \partial H_{X_{k}} + H_{X_{k}} \partial= V_{m} - V_{n} .\]

It follows that $V_{m}$ and $V_{n}$ are chain homotopic. Iterating this argument shows that $V_{i}$ and $V_{j}$ are chain homotopic if $O_{i}$ and $O_{j}$ belong to the same link component.

\end{proof}

In view of the last result, we can think of  $\mathcal{C}_{U\cup \beta}$ as a $\mathbb{F}_{2}[V_{i_{1}},\cdots,V_{i_{l}}]$-module. We can also consider the complex $c\mathcal{C}_{U\cup \beta} \cong \frac{\mathcal{C}_{U\cup \beta}}{V_{i_{1}}=\cdots=V_{i_{l}}}$. It can be easily seen that $\mathcal{F}_{U}$ filtered quasi-isomorphism type of $c\mathcal{C}_{U\cup \beta}$ is also a braid conjugacy class invariant and its homology can be thought of as a $\mathbb{F}_{2}[V]$-module.\\

We will postpone the proof of the following theorem till Section 5.

\begin{theorem}
Filtered quasi-isomorphism type of $\mathcal{C}_{U\cup \beta}$ is a braid conjugacy class invariant.
\end{theorem}

 In the next section we will consider the complex $t\mathbf{C}$ which yields $\mathcal{C}_{U\cup \beta}$ at $t=0$ for briads. There, we will prove a more general version (Theorem \ref{invarianceTALF}) of the above theorem.

\subsection{Crossing change move}

Suppose $L^{+}$ is obtained from $L^{-}$ by changing a negative crossing to positive crossing. The crossing change maps between their $GC^{-}$ version of grid complexes was defined in chapter 6 of \cite{Grid Homology for Knots and Links}. Now, we will study the effect of crossing change map in the complex. First, we discuss the effect of changing a positive crossing to negative in the braid $\beta$. The maps associated with crossing change move will also appear for positive and negative stabilizations.

\begin{figure}
\begin{center}
\includegraphics[scale=0.4]{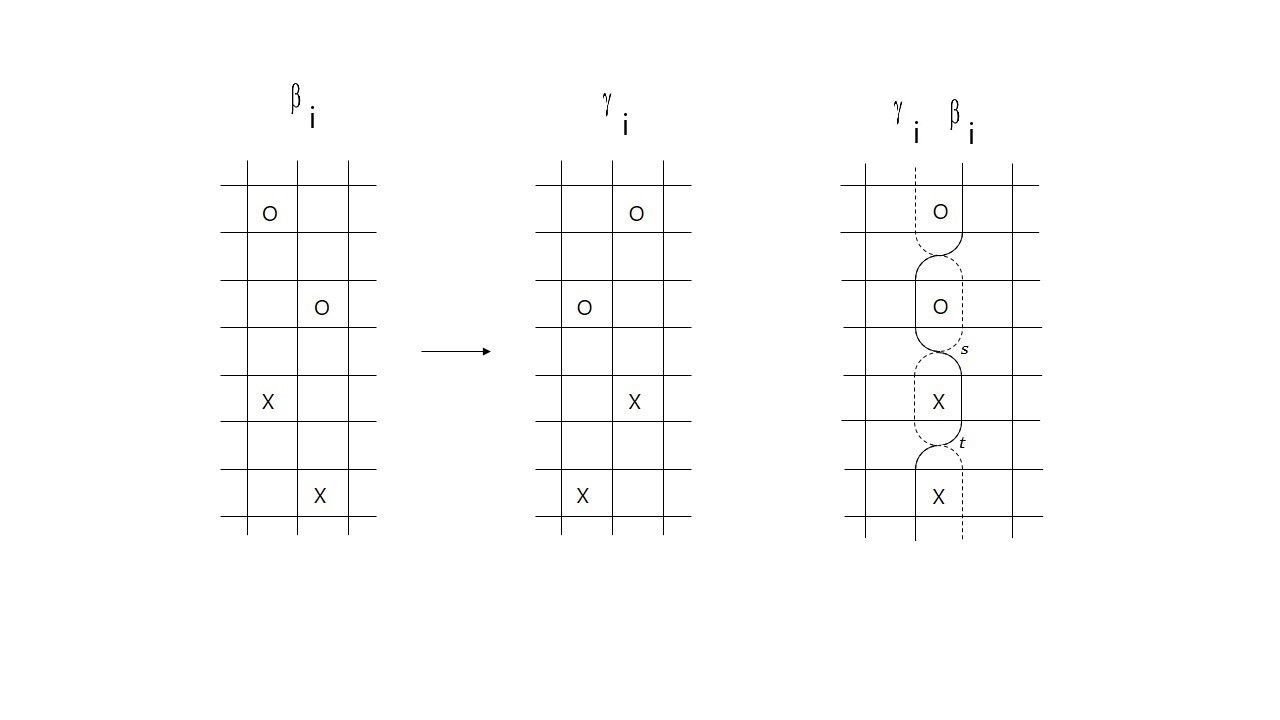}
\caption{Crossing change move}
\end{center}
\end{figure}

Suppose the diagram $D^-$ representing $U \cup \beta^{-}$is  obtained by cross commutation of the two columns from the diagram $D^+$ representing $U \cup {\beta}^{+}$ (See Figure \ref{cross}). The vertical circle $\beta_{i}$ in $D^+$ is replaced by the dotted vertical circle $\gamma_{i}$ in $D^-$ and, they intersect at two points $s$ and $t$. We define the $\mathcal{R}$-module maps $c_{-}$ : $\mathcal{C}_{U \cup {\beta}^{+}}(D^{+}) \rightarrow \mathcal{C}_{U \cup \beta^{-}}(D^{-})$ and $c_{+} : \mathcal{C}_{U \cup \beta^{+}}(D^{-}) \rightarrow \mathcal{C}_{U \cup \beta^{-}}(D^{+})$ for a grid state $x \in S(D^{+})$ and $y' \in  S(D^{-})$
respectively in the following way

\[ c_{-}(x) =\sum_{ y \in S(D^{-})} \sum_{p \in {Pent_{s}}^{o}(x,y), \ \mathbb{X}\cap p= \phi}  V_{3}^{O_{p}(r)}V_{4}^{O_{4}(p)}\cdots V_{n}^{O_{n}(p)} y \]

and,

\[ c_{+}(y') =\sum_{ x' \in S(D^{+})} \sum_{p \in {Pent_{t}}^{o}(y',x'), \ \mathbb{X}\cap p= \phi}   V_{3}^{O_{3}(p)}V_{4}^{O_{p}(r)}\cdots V_{n}^{O_{n}(p)} x' .\]

\begin{figure}
\begin{center}
\includegraphics[scale=0.4]{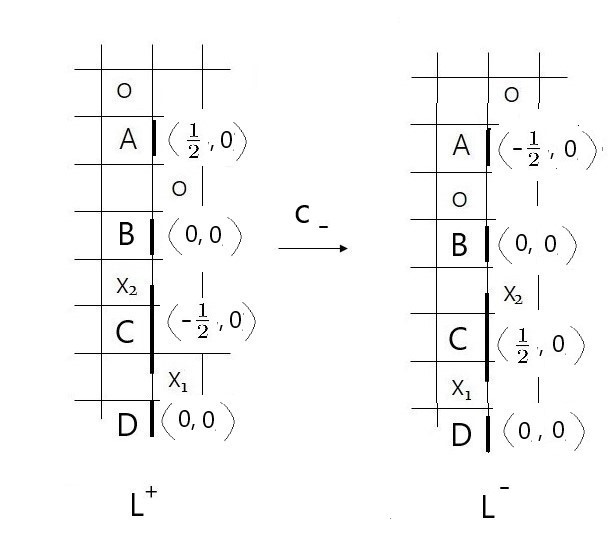}
\caption{Local change in $(A_{\beta},A_{U})$ grading} \label{cricintervals}
\end{center}
\end{figure}

\begin{prop}\label{crossgradechange}
The map $c_{-}$ is $A_{\beta}$ graded and $\mathcal{F}_{U}$ filtered. The map $c_{+}$ is $A_{\beta}$ graded of degree $1$ and $\mathcal{F}_{U}$ filtered.
\end{prop}
\begin{proof}
To compute degrees of $c_-$ and $c_+$ as in Lemma 6.2.1 of \cite{Grid Homology for Knots and Links}, we can compare the pentagons appearing in those maps with rectangles. First notice, that each state $s$ in $D^{+}$ corresponds to a state $\phi(s)$ in $D^{-}$ by mapping the point in $\beta_{i}$ vertical circle to the $\gamma_{i}$ circle. Now, the $4$ X and O markings divide the vertical circles $\beta_{i}$ and $\gamma_{i}$ into four intervals $A$, $B$, $C$ and $D$. We consider states that contains a point in on of the the four intervals $A$, $B$, $C$ and $D$ [See Figure \ref{cricintervals}] between the four special markings. We can make local computations in each case (based on relative position of markings) to compute the difference in gradings between a state and its corresponding state. Now if $y$ is a term appearing in $c_{-}(x)$, the there is a pentagon from $x$ to $y$. To each pentagon from a state in $D^{+}$ to a state in $D^{-}$, we can associate a rectangle [See Figure \ref{explain}] in $D^{+}$ from $x$ to $\phi(y)$. This allows us to compute grading change under $c_{-}$ by the following formula

\[A(x)-A(y)= (A(x)-A(\phi^{-}(y))+ (A(\phi^{-}(y)-A(y)).\]

\begin{figure}[!tbph]
\begin{center}
\includegraphics[scale=0.4]{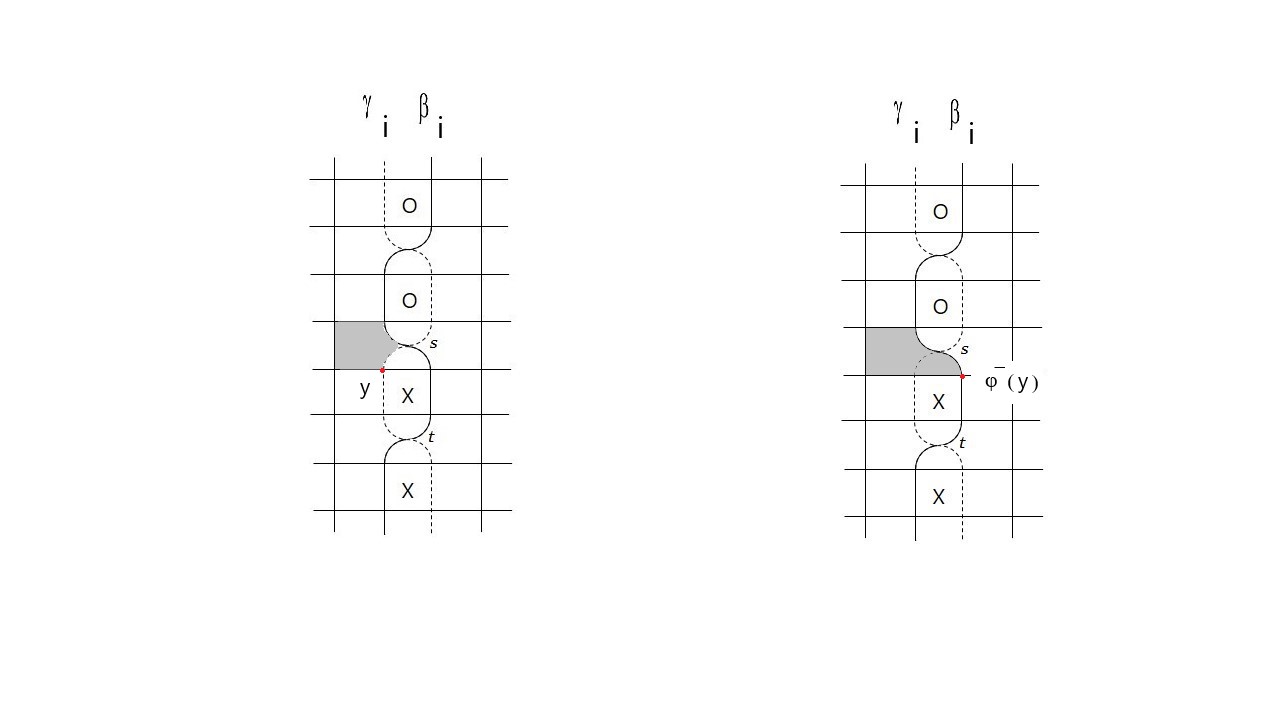}
\caption{A left pentagon and its associated rectangle}\label{explain}
\end{center}
\end{figure}

 The first term $A(x)-A(\phi(y))$ just counts extra contribution in the associated rectangle as the differential preserves degree and filtration, and the second term $A(\phi(y)-A(y)$ computes the local change. In this case, it is easy to see the $\mathcal{F}_{U}$ filtration is unaffected. Similarly, we can compute grading change under $c_{+}$. In our case, there is no change in $\mathcal{F}_{U}$ filtration level for either $c_{-}$ or $c_{+}$ as there no local change in $A_{U}$ grading or any extra $X$ marking belonging to the unknot component in the associated rectangles. So those maps are $\mathcal{F}_{U}$ filtered. Then, the computation of $A_{\beta}$ grading change under $c_{-}$ and $c_{+}$  turns out to be identical to Lemma 6.2.1 of \cite{Grid Homology for Knots and Links}.

\begin{figure}[!tbph]
\begin{center}
\includegraphics[scale=0.4]{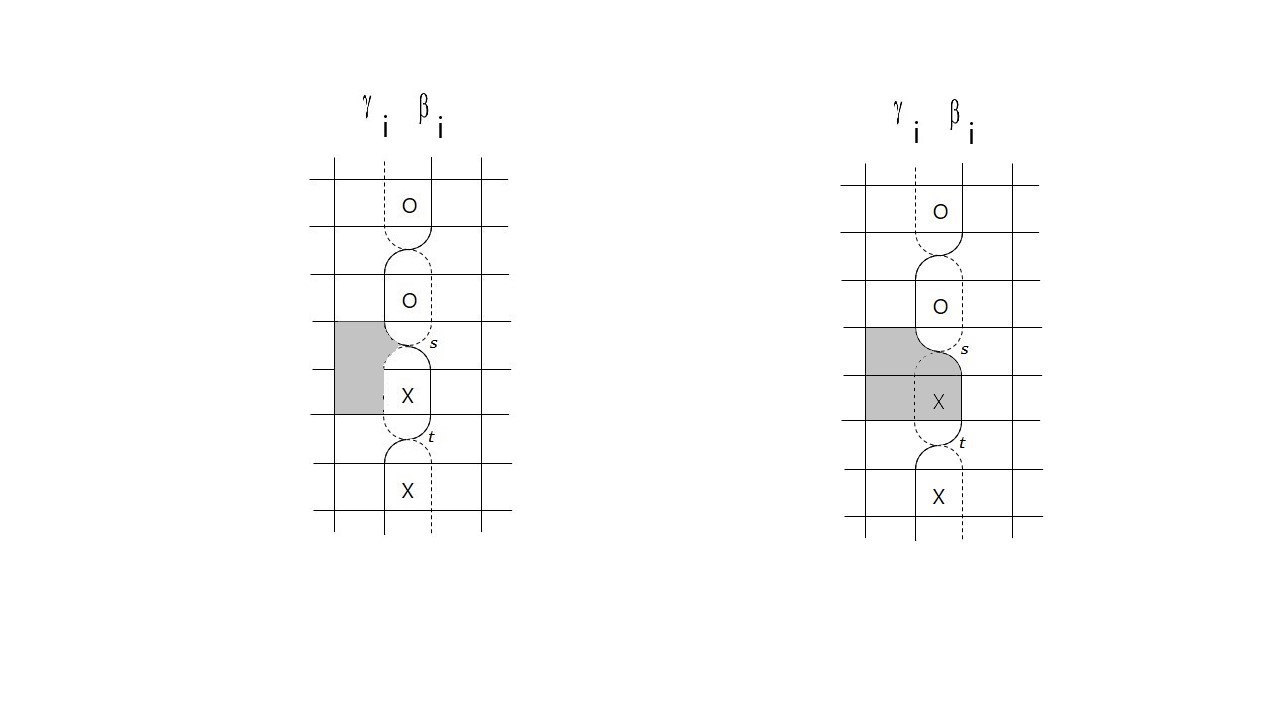}
\caption{Associated rectangle for a left pentagon in interval C}\label{explainC}
\end{center}
\end{figure}

Let $y$ be the term, appearing in $c_{-}(x)$ and assume there is a left pentagon from $x$ to $y$. 

\textbf{Case 1: $y$ is type B}  Since there are no extra markings in the associated rectangles [See Figure \ref{explain}] we have, 

\[A_{\beta}(y)-A_{\beta}(x)=A_{\beta}(\phi^{-}(y))-A_{\beta}(x)+ A_{\beta}(y)- A_{\beta}(\phi^{-}(y))=0\]

and

\[\mathcal{F}_{U}(y)-\mathcal{F}_{U}(x)=\mathcal{F}_{U}(\phi^{-}(y))-\mathcal{F}_{U}(x)+ \mathcal{F}_{U}(y)- \mathcal{F}_{U}(\phi^{-}(y))\leq  \mathcal{F}_{U}(\phi^{-}(y))-\mathcal{F}_{U}(y)=0.\]

\textbf{Case 2: $y$ is type C} In this case, we have an extra X marking belonging to the braid $\beta$ [See Figure \ref{explainC}] in the associated rectangle. So,

\[A_{\beta}(y)-A_{\beta}(x)=A_{\beta}(\phi^{-}(y))-A_{\beta}(x)+ A_{\beta}(y)- A_{\beta}(\phi^{-}(y))= 1-1=0\]

and

\[\mathcal{F}_{U}(y)-\mathcal{F}_{U}(x)=\mathcal{F}_{U}(\phi^{-}(y))-\mathcal{F}_{U}(x)+ \mathcal{F}_{U}(y)- \mathcal{F}_{U}(\phi^{-}(y))\leq  \mathcal{F}_{U}(\phi^{-}(y))-\mathcal{F}_{U}(y)=0.\]

For right pentagons, we consider initial corners in $B$ and $C$ and the computation works similarly. We can also compute grading and filtered degree of $c_+$ using the same technique.

\end{proof}

\begin{prop}
The map $c_{-}$ and $c_{+}$ are chain maps.
\end{prop}
\begin{proof}

Again, consider the $\mathbb{F}_{2}[V_{1},V_{2},\cdots,V_{n}]$-module map $\mathbf{c_{-}}:GC^{-}(D^{+}) \rightarrow GC^{-}(D^{-})$ given by 
\begin{align*}
    \mathbf{c_{-}}(x)= \sum\limits_{y' \in S(D^-)} \sum\limits_{p \in {Pent_{s}}^\circ(x,y'), p \cap \mathbb{X}= \phi}  V_{1}^{O_{1}(p)}V_{2}^{O_{2}(p)}V_{3}^{O_{3}(p)}\cdots V_{n}^{O_{n}(p)} \cdot y'  
\end{align*}
Then, $c_{-}$ is induced map on the quotient $\frac{GC^{-}(D)}{(V_{1}-1)(V_{2}-1)}$. Therefore, $c_{-}$ is a chain map since $\mathbf{c_{-}}$ is a chain map. A similar argument shows that $c_{+}$ is a chain map.

\end{proof}

The chain maps $c_-$ and $c_+$ induce the desired maps $C_-$ and $C_+$ on the homologies.

\begin{prop}\label{cross}
$c_+ \circ c_-$ and $c_- \circ c_+$ are chain homotopic to multiplication by $V_{i}$.
\end{prop}

\begin{proof}
For ${x}_-, {y}_- \in {S}(D^-)$, let $\text{Hex}_{s,t}^\circ({x}_-, {y}_-)$ denote the set of empty hexagons with two consecutive corners at $s$ and at $t$ in the order consistent with the orientation of the hexagon. The set $\text{Hex}_{s,t}^\circ$ for ${x}_+, {y}_+ \in {S}(D^+)$ is defined analogously.

Let $H_- : \mathcal{C}_{U\cup \beta}(D^-) \rightarrow \mathcal{C}_{U\cup \beta}(D^-)$ be the $\mathcal{R}$-module map whose value on any $x_- \in {S}(D^-)$ is
\begin{align*}
    H_-({x}_-)=\sum\limits_{{y}_-\in {S}(D^-)} \sum\limits_{h \in \text{Hex}^\circ_{s,t}({x}_-, {y}_-)} V_{3}^{O_{3}(h)}\cdots V_{n}^{O_{n}(h)}\cdot {y}_-.
\end{align*}
The analogous map $H_+ : \mathcal{C}_{U\cup \beta}(D^+) \rightarrow \mathcal{C}_{U\cup \beta}(D^+)$ is defined in the same way using $\text{Hex}_{s,t}^\circ({x}_+, {y}_+)$.\\

Following the lines of the proof of Proposition 6.1.1. in \cite{Grid Homology for Knots and Links}, we can easily verify that $H_+$ is a chain homotopy between $c_+ \circ c_-$ and the multiplication by $V_{i}$, and that $H_-$ is a chain homotopy between $c_- \circ c_+$ and $V_{i}$, i.e.:

\begin{align*}
    \partial \circ H_+ +  H_+ \circ \partial=c_+ \circ c_- + V_{i},\\
    \partial \circ H_- +  H_- \circ \partial=c_- \circ c_+ + V_{i}.
\end{align*}

\end{proof}

\subsection{Positive stabilization}

\begin{figure}
\begin{center}

\includegraphics[scale=0.12]{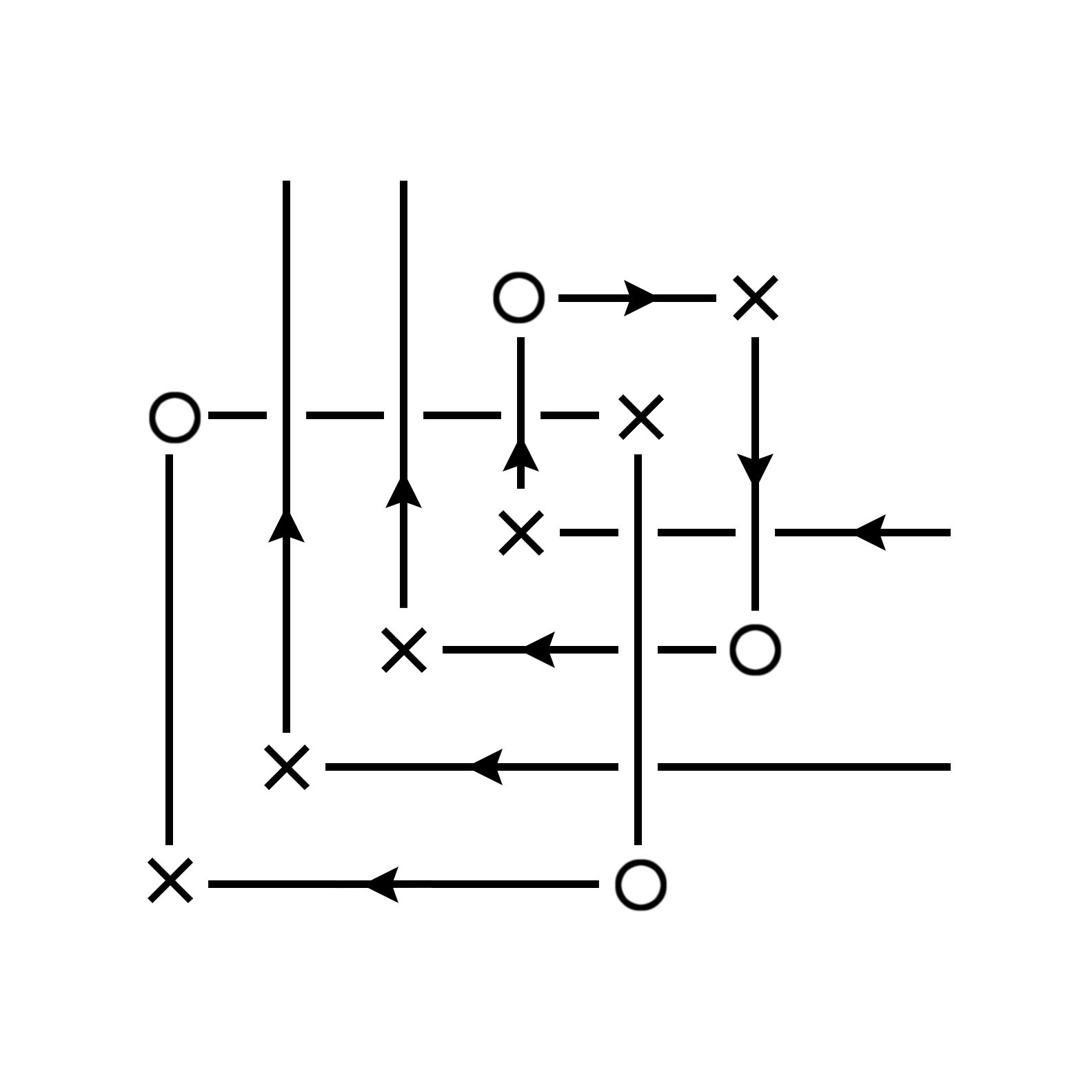}
\caption{Positive stabilization}\label{ps}
\end{center}
\end{figure}

Suppose we have a positive stabilization diagram $D^+$ (See Figure \ref{ps}) and $D^-$ is obtained by cross commutation of the two columns in the left one belonging to the unknot and the other one belonging to the braid $\beta$ as shown in Figure \ref{ps}. It is easy to see that then $D^-$ represents $U \cup \beta$. We define the $\mathcal{R}$-module maps ${PS}_{+}$ : $\mathcal{C}_{U \cup {\beta}_{+stab}}(D^{-}) \rightarrow \mathcal{C}_{U \cup \beta}(D^{+})$ and ${PS}_{-} : \mathcal{C}_{U \cup \beta}(D^{+}) \rightarrow \mathcal{C}_{U \cup \beta_{+stab}}(D^{-})$ for a grid state $x \in S(D^{+})$ and $y' \in  S(D^{-})$
respectively in the following way:
\[ {PS}_{-}(x) =\sum_{ y \in S(D^{-})} \sum_{p \in {Pent_{s}}^{o}(x,y), \ \mathbb{X}\cap p= \phi}  V_{3}^{O_{p}(r)}V_{4}^{O_{4}(p)}\cdots V_{n}^{O_{n}(p)} y \]

\[ {PS}_{+}(y') =\sum_{ x' \in S(D^{+})} \sum_{p \in {Pent_{s}}^{o}(y',x'), \ \mathbb{X}\cap p= \phi}   V_{3}^{O_{3}(p)}V_{4}^{O_{p}(r)}\cdots V_{n}^{O_{n}(p)} x' \]

\begin{figure}
\begin{center}
\includegraphics[scale=0.4]{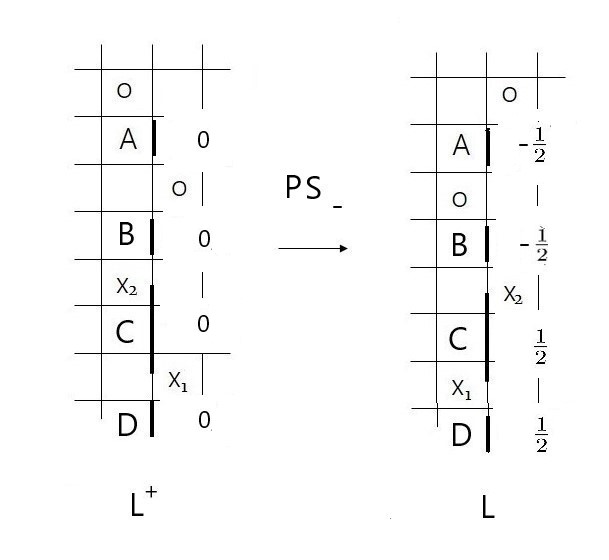}
\caption{Change in local $A_{U}$ grading}\label{psintchange}
\end{center}
\end{figure}

\begin{prop}
The maps ${PS}_{-}$ is $\mathcal{F}_{U}$ filtered and $PS_{+}$ is $\mathcal{F}_{U}$ filtered of degree $\frac{1}{2}$.
\end{prop}
\begin{proof}

We inspect the intervals $A$, $B$, $C$ and $D$ again [See Figure \ref{psintchange}]. Local change in $A_{U}$ can be computed easily using the winding number formula [Equation \ref{windingformula}] for $A_{U}$. Then, we can use the fact that $\mathcal{F}_{U}=-A_{U}$ for grid states to compute filtration change. If $y$ is a term appearing in ${PS}_{-}(x)$ and there is an empty left pentagon $p$ (pentagon to the the left of vertical circle $\beta_{i}$ or $\gamma_{i}$) from $x$ to $y$. Notice that, that the terminal generator $y$ is either of type $B$ or $C$ [ See Lemma 6.2.1 of \cite{Grid Homology for Knots and Links}]. So we can use associated left rectangles for the left pentagons to compute grading change as in Proposition \ref{crossgradechange}.

\begin{figure}[!tbph]
\begin{center}
\includegraphics[scale=0.4]{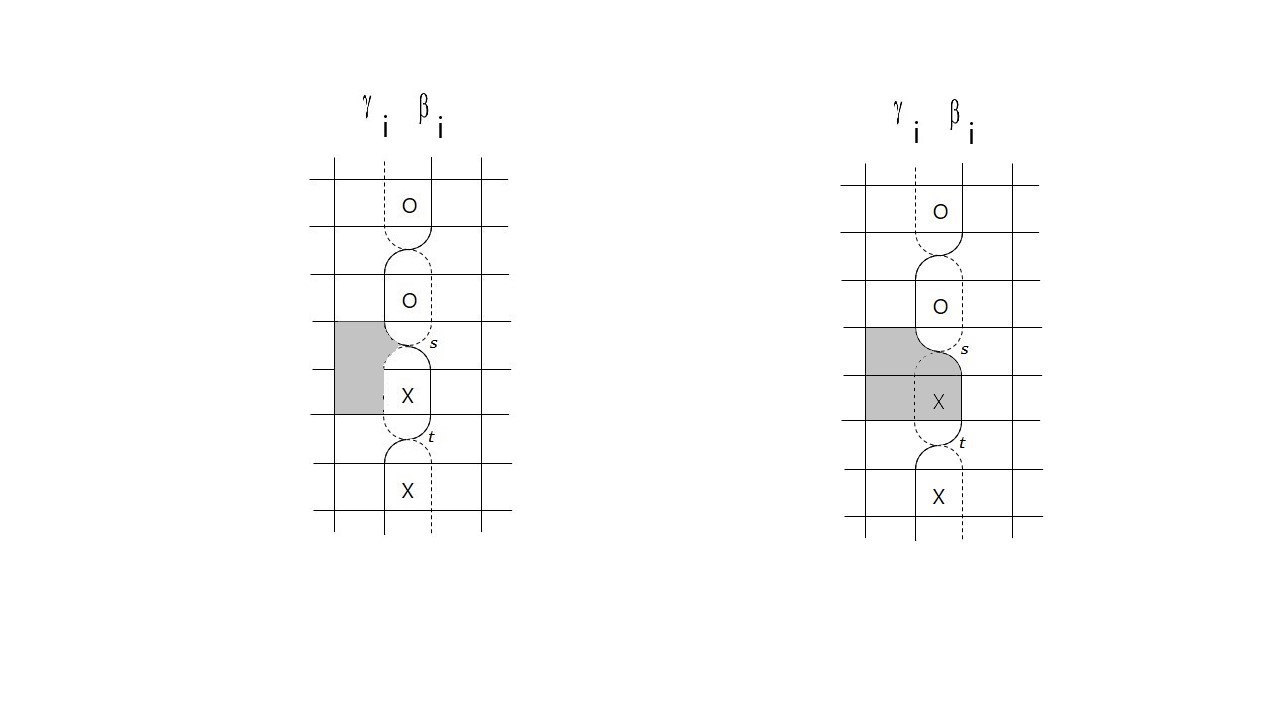}
\caption{Associated rectangle for a left pentagon in interval C}\label{explainC}
\end{center}
\end{figure}

\textbf{Case 1: $y$ is type B}  Since there are no extra markings in the associated rectangles [See Figure \ref{explain}] we have, \[\mathcal{F}_{U}(y)-\mathcal{F}_{U}(x)=\mathcal{F}_{U}(\phi^{-}(y))-\mathcal{F}_{U}(x)+ \mathcal{F}_{U}(y)- \mathcal{F}_{U}(\phi^{-}(y))\leq  \mathcal{F}_{U}(\phi^{-}(y))-\mathcal{F}_{U}(y)=-\frac{1}{2}\]

\textbf{Case 2: $y$ is type C} In this case, we have an extra X marking belonging to the unknot [See Figure \ref{explainC}] in the associated rectangle. So,
\[\mathcal{F}_{U}(y)-\mathcal{F}_{U}(x)=\mathcal{F}_{U}(\phi^{-}(y))-\mathcal{F}_{U}(x)+ \mathcal{F}_{U}(\phi^{-}(y))-\mathcal{F}_{U}(y)\leq  -1 + \mathcal{F}_{U}(\phi^{-}(y))-\mathcal{F}_{U}(y)=-1+\frac{1}{2}=-\frac{1}{2}.\] 

Similarly for a right pentagon (pentagon to the the right of vertical circle $\beta_{i}$ or $\gamma_{i}$), we compare it with a right rectangle. Here the initial corner is either of type $B$ or type $C$. In each case, we get filtration change $=0$. Therefore, $PS_{-}$ is $\mathcal{F}_{U}$ filtered of degree $0$.\\

Now, let $y'$ be a term appearing in ${PS}_{+}(x')$.\\ 

\textbf{Case 1: $y'$ is type B}
From Figure \ref{explainplus}, we observe that the associated rectangle has one less O marking  belonging to the unknot. So we have, \[\mathcal{F}_{U}(y')-\mathcal{F}_{U}(x')=\mathcal{F}_{U}(\phi(y'))-\mathcal{F}_{U}(x')+ \mathcal{F}_{U}(y')- \mathcal{F}_{U}(\phi(y'))\leq  1+\mathcal{F}_{U}(y')-\mathcal{F}_{U}(\phi(y'))=\frac{1}{2}\]

\textbf{Case 2: $y'$ is type C}
There are no additional markings. So we have, \[\mathcal{F}_{U}(y')-\mathcal{F}_{U}(x')=\mathcal{F}_{U}(\phi(y'))-\mathcal{F}_{U}(x')+ \mathcal{F}_{U}(y')- \mathcal{F}_{U}(\phi(y'))\leq \mathcal{F}_{U}(y')-\mathcal{F}_{U}(\phi(y'))=\frac{1}{2}\]

\textbf{Case 3: $y'$ is type D}
Again there are no additional markings. Therefore,
 \[\mathcal{F}_{U}(y')-\mathcal{F}_{U}(x')=\mathcal{F}_{U}(\phi(y'))-\mathcal{F}_{U}(x')+ \mathcal{F}_{U}(y')- \mathcal{F}_{U}(\phi(y'))\leq \mathcal{F}_{U}(y')-\mathcal{F}_{U}(\phi(y'))=\frac{1}{2}\]

\textbf{Case 4: $y'$ is type A}
Again the associated rectangle has one less O marking  belonging to the unknot. So we have, \[\mathcal{F}_{U}(y')-\mathcal{F}_{U}(x')=\mathcal{F}_{U}(\phi(y'))-\mathcal{F}_{U}(x')+ \mathcal{F}_{U}(y')- \mathcal{F}_{U}(\phi(y'))\leq 1+\mathcal{F}_{U}(y')-\mathcal{F}_{U}(\phi(y'))=\frac{1}{2}\]

Hence, the map $PS_{+}$ is $\mathcal{F}_{U}$ filtered of degree $\frac{1}{2}$.

\end{proof}

\begin{figure}[!tbph]
\begin{center}
\includegraphics[scale=0.4]{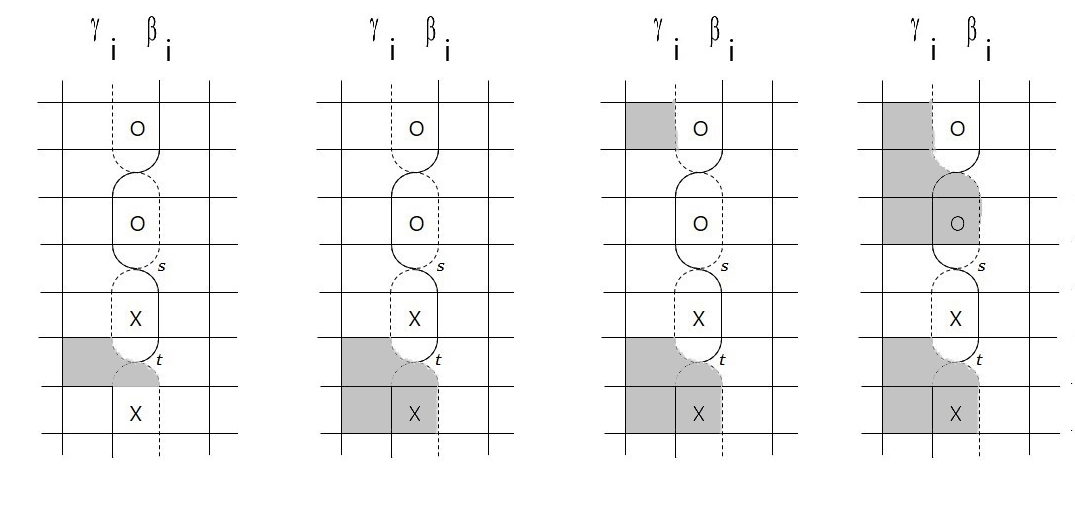}
\caption{Associated rectangle for interval A, B, C and D in the map $PS_+$}\label{explainplus}
\end{center}
\end{figure}

\begin{prop}
The maps ${PS}_{-}$ and ${PS}_{+}$ are chain maps.
\end{prop}
\begin{proof}
This same argument shows that $\partial \circ{ PS}_{-} + {PS}_{-} \circ \partial=0$ and $\partial \circ{ PS}_{+} + {PS}_{+} \circ \partial=0$. 
\end{proof}

\begin{prop}
 $PS_{-}$ and ${PS}_{+}$ are quasi-isomorphisms.
\end{prop}

\begin{proof}

Let us revisit the homotopy equivalence maps from the previous section again. In this case, the same argument shows
 
\begin{align*}
    \partial \circ H_+ +  H_+ \circ \partial={PS}_{+}\circ{PS}_{-} + 1,\\
    \partial \circ H_- +  H_- \circ \partial={PS}_{-}\circ{PS}_{+} + 1.
\end{align*}

The conclusion follows.
\end{proof}

\subsection{Refinement of $\theta$}

Let $x_{4}$ be the distinguished cycle consisting of northeast corners of X-markings in $D$ as described before. Let $V$ be the subset of homology generated by $V_{i}$'s. Any class belonging in $V$ is said to be in $V$-image.

\begin{definition}

Define $\eta(\beta) = \min \ \{ \ k \ | \ [x_{4}]$ is in the $V$-image in  $H_{*}( {\mathcal{F}_{U}}^{k}(\mathcal{C}_{U \cup \beta})   \} $.
\end{definition}

Now we are ready to prove Theorem \ref{theorem2}.

\begin{proof} [Proof of Theorem \ref{theorem2}]

From \cite{lossbraid}, we know that $H_{top}( {\mathcal{F}_{U}}^{bot}(\mathcal{C}_{U \cup \beta}))$ is generated by $[x_{4}]$. So we can reinterpret the definition as  $ \eta (\beta)= \min \ \{ \ k \ | \  \text{ the natural inclusion map } i: H_{top}( {\mathcal{F}_{U}}^{bot}(\mathcal{C}_{U \cup \beta})) \rightarrow\frac{ H_{*}({\mathcal{F}_{U}}^{k}( \mathcal{C}_{U \cup \beta}))}{V} \text{ is trivial }     \} $. It follows that $\eta (\beta) $ is a braid conjugacy class invariant. Since ${PS}_{-}$ and ${PS}_{+}$ are filtered quasi-isomorphisms, the second claim follows from the filtered degrees of the maps ${PS}_{-}$ and ${PS}_{+}$. 

\end{proof} 

Using the above theorem, we know that $\eta(\beta)$ is non decreasing under positive stabilization. So it is possible to define a transverse invariant $\bar{\eta}(\beta)$  by taking minimum over all braid representatives. 

\begin{proof} [Proof of Theorem \ref{theorem3}]

From  \cite{lossbraid}, we know that there is a inclusion map $I: HFK^{-}(m(\beta))  \rightarrow HFK^{-,2}({\mathcal{H}}_{4})$ sending $\theta(\mathcal{T})$ to the distinguished class $[x_{4}] \in HFK^{-,2}({\mathcal{H}}_{4})$. Therefore, if $\hat{\theta}(\mathcal{T}) \neq 0$ then $[x_{4}]$ can't be in the $V$ image in any filtration level in $HFK^{-,2}\cong H(\mathcal{C}_{U\cup \beta})$ and vice-versa. Hence, it follows that $\eta(\beta)=\infty$ if and only if  $\hat{\theta}(\mathcal{T})\neq 0$. Also since $[x_{4}]$ generates the top Maslov grading in the bottom filtration in $HFK^{-,2}$. 
From the winding number formula [Equation \ref{windingformula}], we have $bot=A_{U}(x^{+})=-N+\frac{1}{8}4(N+1)-\frac{1}{2}=-\frac{N}{2}$. It follows that $\eta(\beta) \geq bot = -\frac{N}{2}$. Also since $\mathcal{F}_{U}^{\frac{N}{2}}(\mathcal{C}_{U\cup \beta})=\mathcal{C}_{U\cup \beta}$, it is obvious that  $\frac{N}{2} \geq \eta(\beta)$.

\end{proof}

It also follows that the transverse invariant $\bar{\eta}(\beta)$ is atleast as strong as $\hat{\theta}$. 

\begin{prop}
If $\beta^{+}$ is obtained from $\beta^{-}$ by changing a negative crossing to positive crossing then,  $\eta(\beta^{+}) \geq \eta(\beta^{-})$.

\end{prop}

\begin{proof}

It can be shown that the crossing change map $C_{-}$ sends the cycle $[x_{4}]$ in $H(\mathcal{C}_{U \cup {\beta}^{+}})$ to $[x_{4}]$ in $H(\mathcal{C}_{U \cup {\beta}^{-}})$. Since $C_{-}$  is filtered of degree $0$, the conclusion follows.
\end{proof}

\subsection{Negative stabilization} 

\begin{figure}
\begin{center}
\includegraphics[scale=0.12]{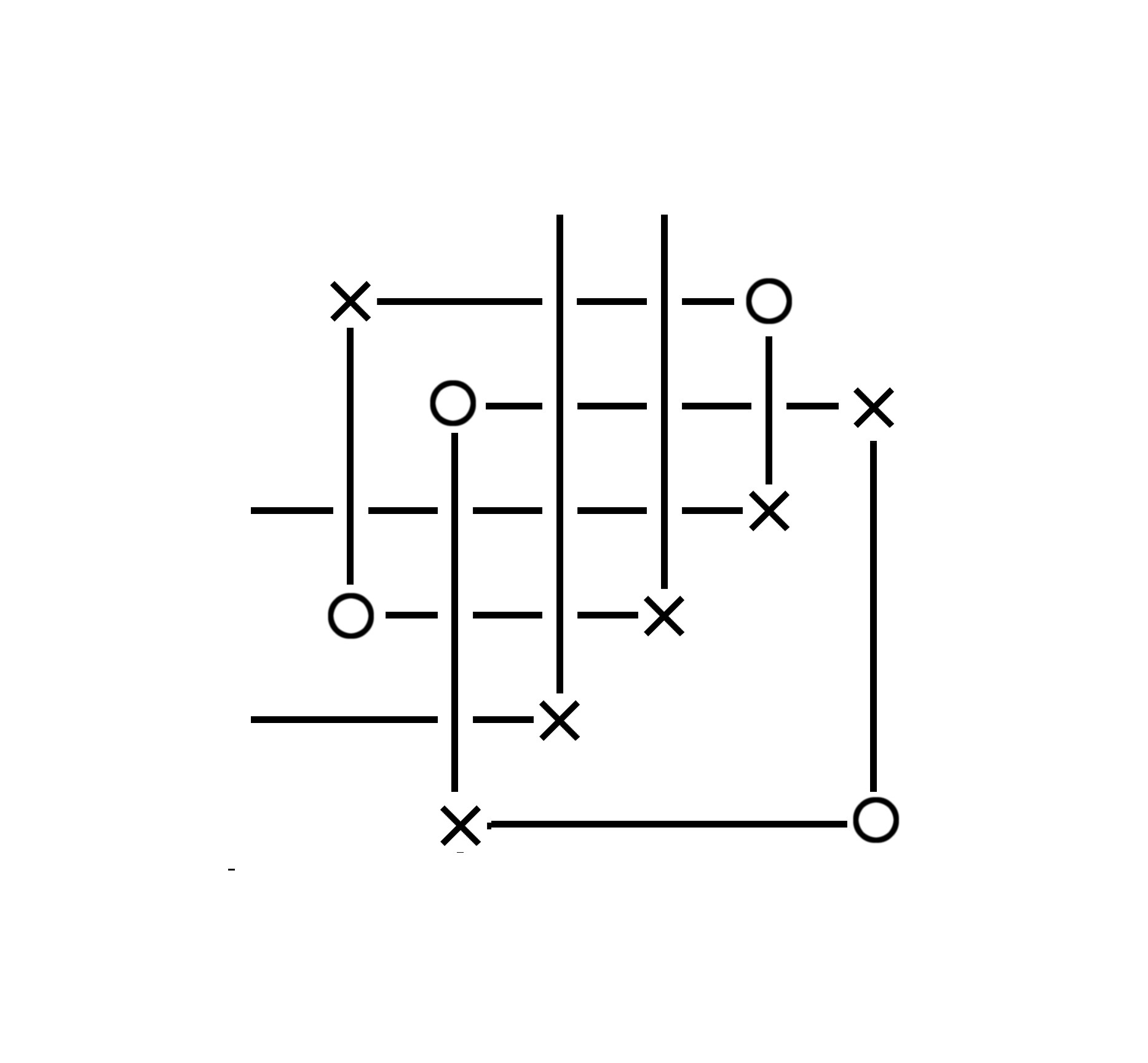}
\caption{Negative stabilization}\label{ns}
\end{center}

\end{figure}

Now lets consider a negative stabilization diagram as in Figure \ref{ns}. Then we can define analogous maps ${NS}_{-}$ : $\mathcal{C}_{U \cup {\beta}_{-stab}}(D^{+}) \rightarrow \mathcal{C}_{U \cup \beta}(D^{-})$ and ${NS}_{+} : \mathcal{C}_{U \cup \beta}(D^{-}) \rightarrow \mathcal{C}_{U \cup \beta_{-stab}}(D^{+})$ be the analogous $\mathcal{R}$-module maps. Again the same argument shows that these are chain maps. However, it is easily checked that they are not filtered quasi-isomorphisms through the next proposition.

\begin{prop}
$NS_{+} \circ NS_{-}$ is chain homotopic to $V_{i}$.
\end{prop}

\begin{proof}

This also follows from the discussion in Proposition \ref{cross}.

\end{proof}

Hence, unlike positive stabilizations these maps are not quasi-isomorphisms. In fact, in the next proposition, we show that negative stabilizations have a special value.

\begin{figure}[!tbph]
\begin{center}
\includegraphics[scale=0.15]{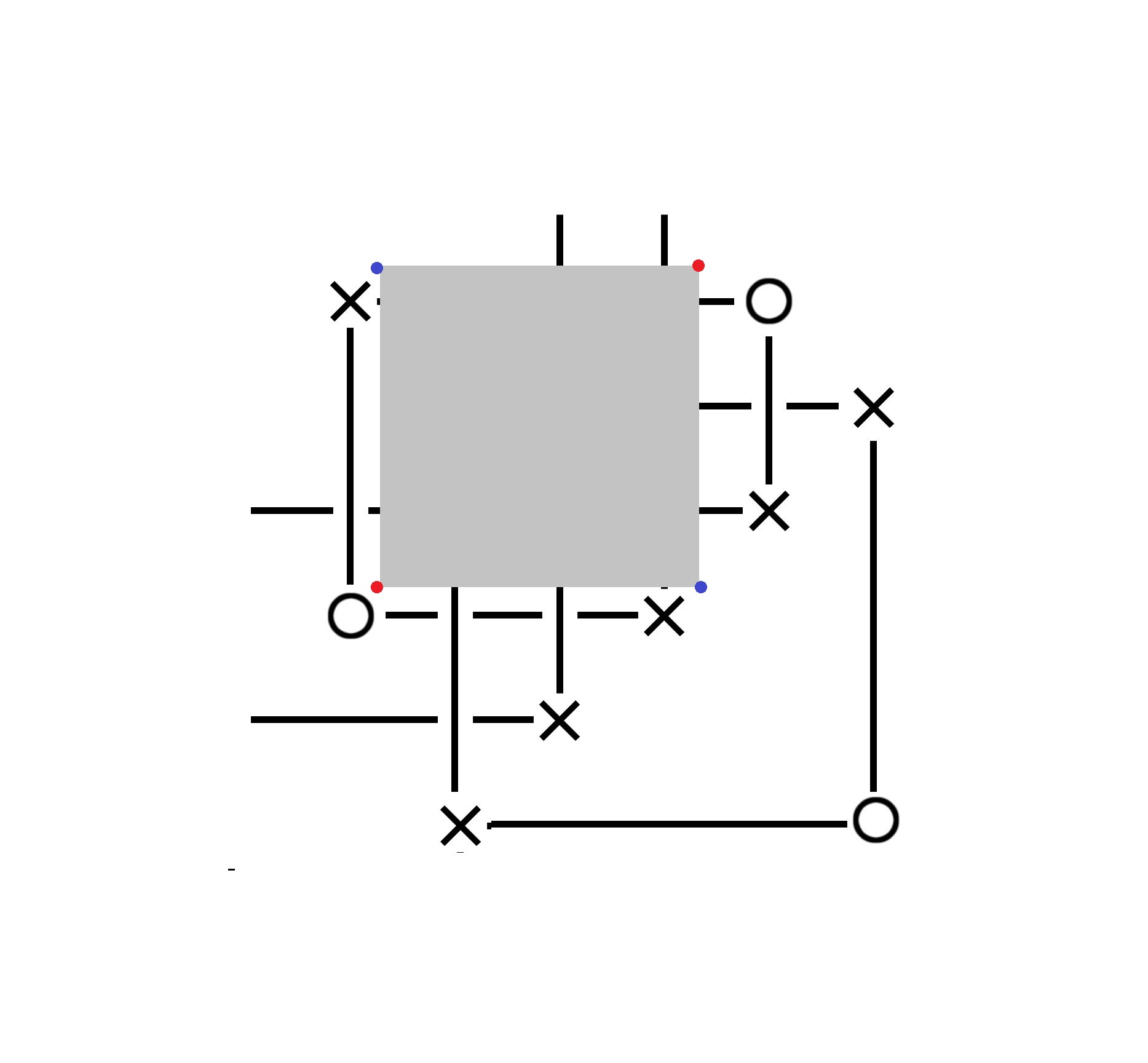}
\caption{Rectangle to the distinguished state in a negative stabilization}\label{nstrivial}
\end{center}

\end{figure}

\begin{prop}

$\eta(\beta_{-stab})= \frac{-N+1}{2}$.

\end{prop}

\begin{proof}

Let $r$ be red colored state in Figure \ref{nstrivial} where $x_{4}$ is depicted by the blue colored state. It is easy to see that $r\in \mathcal{F}^{\frac{-N+1}{2}}(\mathcal{C}_{U\cup\beta})$. By considering the rectangles from $r$, we find that either they must contain a O marking in the $\beta$ component or they connect to the distinguished state $x_{4}$. So, $\partial r = x_{4} + V_{i}(..)$. It follows that $[x_{4}]$ is $V$-image in $\mathcal{F}^{\frac{-N+1}{2}}(\mathcal{C}_{U\cup\beta})$ and the conclusion follows.
\end{proof}

\begin{remark}

These properties of $\eta(\beta)$ are very similar to the Kapppa invariant \cite{saltz} in Khovanov homology which served as a motivation for defining $\eta$. It is not clear at this point how one can compute the transverse $\bar{\eta}(\beta)$.
\end{remark}

\pagebreak

\section{A $t$-modified annular chain complex}

Now we will generalize the construction of the braid grid complex from the last section to define the $t$-modified annular chain complex that will help us reformulate the annular invariant $\mathscr{A}_{L}(t)$. Then, we will investigate how the invariant behaves under crossing change and stabilization.   

\subsection{Definition}
Let $D$ be a grid diagram of annular link $U \cup L$. Let $\mathbb{X}=\{X_{1},X_{2},X_{3},\cdots,X_{n} \}$  and $\mathbb{O}=\{O_{1},O_{2},O_{3},\cdots,O_{n} \}$ be the sets of X markings and O markings respectively where $X_{1}, X_{2}, O_{1} \text{ and } O_{2}$ represent the markings of the unknot $U$. For  $0 \leq t=2\frac{p}{q} \leq 2$ where $p,q$ are co-prime non-negative integers, we define the following modified link Floer complex.

\begin{definition}
   Define $t\mathbf{C}(D) = \mathbb{F}_{2}[V_{1},V_{2},V_{3},\cdots,V_{n-1}]$ module over grid states $ S(D)$ and  for a $x \in S(D)$,
 \[ {\partial}_{t}x := \mathlarger{\sum\limits_{y\in S(D)} \sum\limits_{r \in Rect^{o}(x,y), r \cap \mathbb{X} = \phi}} V_{1}^{pO_{1} (r)}V_{2}^{pO_{2} (r)}V_{3}^{(q-p)O_{3}(r)}\cdots V_{n-1}^{(q-p)O_{n-1}(r)}V_{1}^{(q-p)O_{n}(r)} y .\]  
\end{definition}

We first show that ${\partial}_{t}$ is indeed a differential.

\begin{prop}

 ${\partial}_{t} \circ {\partial}_{t}=0$.
\end{prop}

\begin{proof}
Let  $ \mathscr{C}_{U\cup L}$  be the $\mathbb{F}_{2}[ W_{1}, W_{2}, \cdots, W_{n}]$ module over grid states $S(D)$ 

\[ {\partial}_{\mathbb{X}}x := \mathlarger{\sum\limits_{y\in S(D)} \sum\limits_{r \in Rect^{o}(x,y), r \cap \mathbb{X} = \phi}} W_{1}^{p(q-p)O_{1}(r)} W_{2}^{p(q-p)O_{2}(r)}\cdots W_{1}^{p(q-p)O_{n}(r)} y\]  
 
$( \mathscr{C}_{U\cup L}$,${\partial}_{\mathbb{X}})$ is a chain complex obtained  $GC^{-}(U\cup L)$ by change of variables $V_{i} \rightarrow W_{i}^{p(q-p)}$. Now, consider the quotient complex  $\frac{\mathscr{C}_{U\cup L}}{W_{1}^{q-p} - W_{n}^{p}}$. After setting, $V_{1}=W_{1}^{q-p}=W_{n}^{p}$, $V_{2}=W_{2}^{q-p}$ and $V_{i}=W_{i}^{p}$ for $i>2$, we observe that the differential ${\partial}_{\mathbb{X}}$ becomes ${\partial}_{t}$ in the quotient. Hence, $t\mathbf{C} \cong \frac{\mathscr{C}_{U\cup L}}{W_{1}^{q-p} - W_{n}^{p}} $ and ${\partial}_{t}$ is just the restriction of ${\partial}_{\mathbb{X}}$ to the quotient. The conclusion follows easily.
 
\end{proof}

$t\mathbf{C}$ is not Maslov graded. We define a function, $\mathscr{F}_{t}(x)=\frac{ pA_{U}(x)+ (q-p)A_{L}(x)}{q}$ for $x\in S(D)$ which is extended to $t\mathbf{C}$ by setting $\mathscr{F}_{t}(V_{i})= -\frac{1}{q}$. Similarly define, $\mathit{F}_{t}(x)=\frac{ (2p-2q)A_{U}(x)- \frac{p}{2} A_{L}(x)}{q}$ for $x\in S(D)$ and $\mathit{F}_{t}(V_{i})= -\frac{1}{q}$ \\

\begin{prop}
$t\mathbf{C}$  is $\mathscr{F}_{t}$ graded and $\mathit{F}_{t}$ filtered.
\end{prop}

\begin{proof}

Let $y$ be a state appearing in the expansion of differential of $x$. Then, 

\begin{align*}
 \mathscr{F}_{t}(V_{1}^{pO_{1} (r)}V_{2}^{pO_{2} (r)}V_{3}^{(q-p)O_{3}(r)}\cdots V_{n-1}^{(q-p)O_{n-1}(r)}V_{1}^{(q-p)O_{n}(r)} y)-\mathscr{F}_{t}(x)  
 \end{align*}
\begin{align*}
=\mathscr{F}_{t}(y')-\mathscr{F}_{t}(x)-\frac{1}{q}(pO_{U}(r)+(q-p)O_{L}(r))=0. 
\end{align*}
And,
\begin{align*}
\mathit{F}_{t}(V_{1}^{pO_{1} (r)}V_{2}^{pO_{2} (r)}V_{3}^{(q-p)O_{3}(r)}\cdots V_{n-1}^{(q-p)O_{n-1}(r)}V_{1}^{(q-p)O_{n}(r)} y)-\mathit{F}_{t}(x)  
\end{align*}
\begin{align*}
=\mathit{F}_{t}(y')-\mathit{F}_{t}(x)-\frac{1}{q}(pO_{U}(r)+(q-p)O_{L}(r))= \frac{1}{q}( (p-2q)O_{U}(r)+(\frac{p}{2}-q)O_{L}(r))\leq 0.
\end{align*}

\end{proof}

The homology of $t\mathbf{C}$ will be denoted by  $tH_{*}(\mathbf{C})$.

\subsection{Invariance of $t\mathbf{C}$}

The goal of this section is to prove the following theorem
\begin{theorem}\label{invarianceTALF}
${\mathit{F}}_{t}$ filtered quasi-isomorphism type of $t\mathbf{C}$ is an annular link invariant.
\end{theorem}
The proof is identical to the invariance proof of grid homology. We sketch the details for completeness. The key observation is that we need to check invariance under commutation moves and stabilization moves of the component $L$.

\subsubsection{Commutation move}

\begin{figure}
\begin{center}
\includegraphics[scale=0.10]{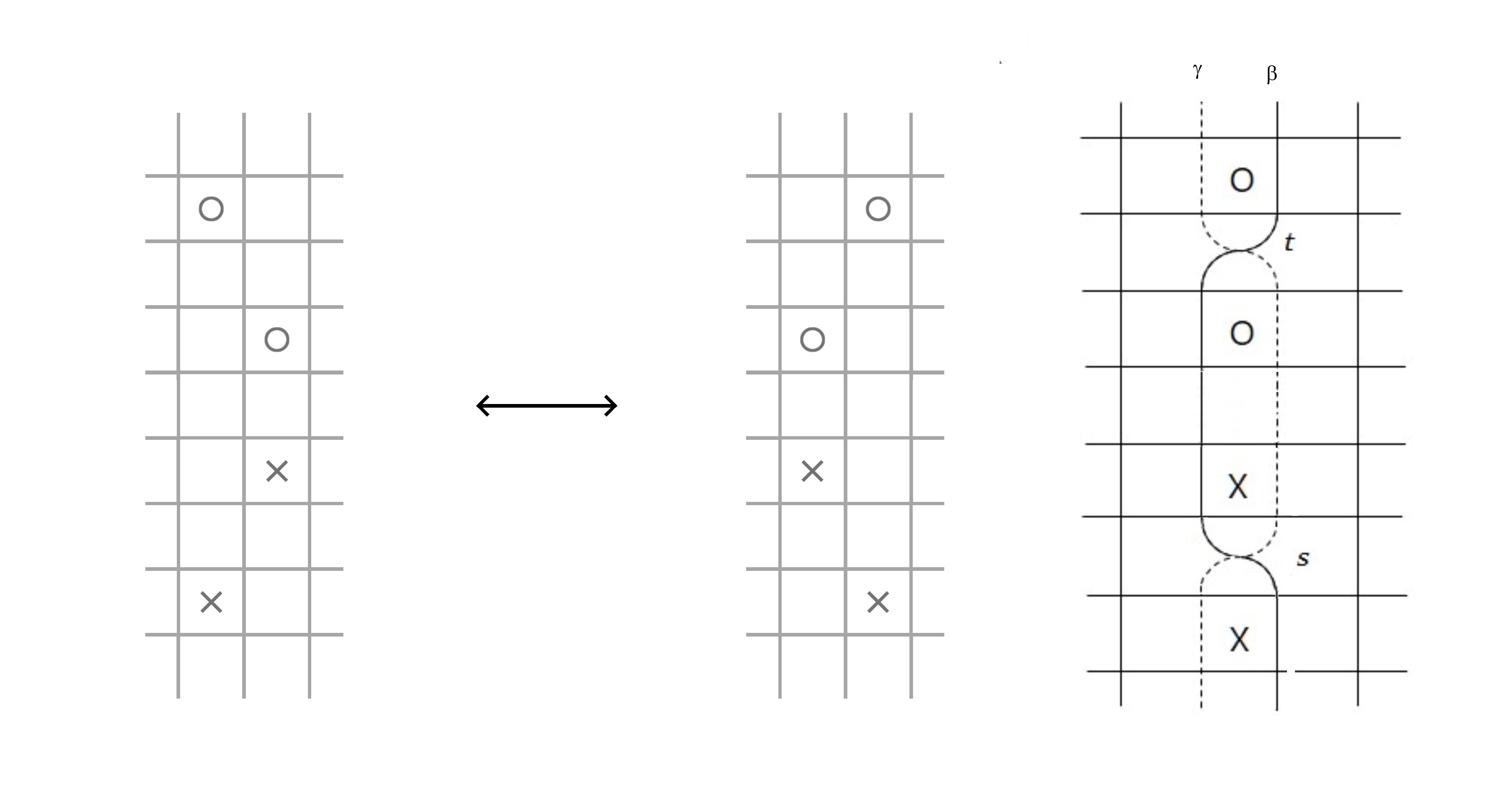}
\end{center}
\caption{Vertical curves in a commutation move}\label{commutfig}
\end{figure}

Consider two grid diagrams $D$ and $D'$ with grid number $n$ that differ by a commutation move. We can represent these two diagrams in the same picture so that the $X$ and $O$ markings are fixed, and two of the vertical circles are curved [See Figure \ref{commutfig}]. Denote the horizontal circles of $D$ by $\alpha = \{\alpha_1, ...\alpha_n\}$ and its vertical circles by $\beta = \{\beta_1, ...\beta_n\}$. Then the set of horizontal circles of $D'$ is also $\alpha$ and its vertical circles are given by $\gamma = \{\beta_1, ..., \beta_{i-1}, \gamma_i, \beta_{i+1}, ..., \beta_n\}$. The vertical circles  $\beta_i$ and $\gamma_i$ intersect at two points. The complement of $\beta_i \cup \gamma_i$ in the grid, consists of two bigons. Consider the one, of which the western boundary is a part of $\beta_i$, and the eastern boundary is a part of $\gamma_i$. We label the two intersection points by $s$ and $s'$ so that $s$ is the southern vertex and $s'$ is the northern vertex of that bigon.\\

Define the $\mathcal{R}$-module map $P\colon t\mathbf{C}(D) \rightarrow t\mathbf{C}(D')$ by the formula:
\begin{align*}
    P(x)= \sum\limits_{y' \in S(D')} \sum\limits_{\Pi \in {Pent}^\circ(x,y'), \Pi \cap \mathbb{X}= \phi}  V_{1}^{pO_{1} (\Pi)}V_{2}^{pO_{2} (\Pi)}V_{3}^{(q-p)O_{3}(\Pi)}\cdots V_{n-1}^{(q-p)O_{n-1}(\Pi)}V_{1}^{(q-p)O_{n}(\Pi)} \cdot y'  
\end{align*}

\begin{lemma}
$P$ is $\mathscr{F}_{t}$ graded and $\mathit{F}_{t}$ filtered.
\end{lemma}

\begin{proof}

Suppose $\Pi$ is an empty pentagon from $x$ to $y'$ in the expansion of $P(x)$. Then,

\begin{align*}
      A_{U}(x)- A_{U}(y')=  O_{U}(\Pi) \ \text{and} \
      A_{L}(x)- A_{L}(y')=  O_{L}(\Pi)
\end{align*}

The conclusion follows by taking the sum with proper weights.

\end{proof}

The following proposition can be proved analogously. 
\begin{proposition}
The map $P$ is a chain map.
\end{proposition}

Now, we define an analogous $\mathcal{R}$-module homomorphism $P^\prime\colon t\mathbf{C}(D') \rightarrow t\mathbf{C}(D)$.
For a grid state ${x}\in {S}(D^\prime)$, let
\begin{align*}
    P^\prime({x}^\prime)=\sum\limits_{{y}\in {S}(D)} \left(\sum\limits_{p \in \text{Pent}^\circ({x}^\prime, {y})} V_{1}^{pO_{1} (\Pi)}V_{2}^{pO_{2} (\Pi)}V_{3}^{(q-p)O_{3}(\Pi)}\cdots V_{n-1}^{(q-p)O_{n-1}(\Pi)}V_{1}^{(q-p)O_{n}(\Pi)}  \right)\cdot {y}.
\end{align*}

Again. We will show that the two maps $P$ and $P^\prime$ are homotopy inverses of each other. 

Define the $\mathcal{R}$-module homomorphism $H\colon t\mathbf{C}(D)\rightarrow t\mathbf{C}(D)$  for each ${x}\in {S}(D)$ by the formula:

\begin{align*}
     H({x})=\sum\limits_{{y}\in {S}(D)} \left(\sum\limits_{h \in \text{Hex}^\circ({x}, {y})}V_{1}^{pO_{1} (h}V_{2}^{pO_{2} (h)}V_{3}^{(q-p)O_{3}(h)}\cdots V_{n-1}^{(q-p)O_{n-1}(h)}V_{1}^{(q-p)O_{n}(h)}  \right)\cdot {y}.
\end{align*}

An analogous map $H^\prime\colon t\mathbf{C}(D^\prime)\rightarrow t\mathbf{C}(D^\prime)$ can be defined by counting empty
hexagons from ${S}(D^\prime)$ to itself.

\begin{lemma}
The map $H\colon t\mathbf{C}(D)\rightarrow t\mathbf{C}(D)$ provides a chain homotopy from the chain map $P^\prime \circ P$ to the identity map on $t\mathbf{C}(D)$.
\end{lemma}
\begin{proof}
We can show that $H$ increases the $\mathscr{F}_{t}$-grading by $1$ by following the same argument. 

To have that $H$ is a chain homotopy from $P^\prime \circ P$ to the identity map on $t\mathbf{C}(D)$, we need to verify the following identities which are true by the same argument.
\begin{align*}
    \partial_{t}^{-}\circ H +  H \circ \partial_{t}^{-}= {Id} -P^\prime \circ P, \quad \text{ that is,}
\end{align*}
\begin{align*}
    (\partial_{t}^{-}\circ H +  H \circ \partial_{t}^{-}+P^\prime \circ P)({x})=  {x} \quad \text{ for any } {x}\in {S}(D).
\end{align*}

\end{proof}

Now by putting it all together, we get -

\begin{theorem}
Let $D$ and $D^\prime$ be two grid diagrams that differ by a commutation move. Then
\begin{align*}
    H(t\mathbf{C})(D)\cong H(t\mathbf{C})(D^\prime).
\end{align*}
\end{theorem}

\subsubsection{Stabilization moves}

{Let $D$ be a grid diagram. By performing a stabilization of type $X \colon SW$, we get the diagram $D^\prime$ . Number the markings in the way that $O_1$ is the newly-introduced $O$-marking, $O_2$ is in the consecutive row below $O_1$, $X_1$ and $X_2$ lie in the same row as
$O_1$ and $O_2$, respectively, i.e.}
\begin{tabular}{ c|c } 
 $X_1$ & $O_1$\\
 \hline
 $ $ & $X_2$\\
\end{tabular}.

{Denote $c$ the intersection point of the new horizontal and vertical circles in $D^\prime$. Considering this point, we can partition the grid states of the stabilized diagram $D^\prime$ into two parts, depending on whether or not they contain the intersection point $c$. Define the sets ${I}(D^\prime)$ and ${N}(D^\prime)$ so that ${x} \in {I}(D^\prime)$ if $c$ is included in ${x}$, and ${x} \in {N}(D^\prime)$ if $c$ is not included in ${x}$. Now ${S}(D^\prime)$  = ${I}(D^\prime) \cup {N}(D^\prime)$ gives a decomposition of $t\mathbf{C}(D^\prime) \cong I \oplus N$, where $I$ and $N$ denote the $\mathcal{R}$-modules spanned by the grid states of ${I}(D^\prime)$ and ${N}(D^\prime)$ respectively.}

{There is a one-to-one correspondence between grid states of ${I}(D^\prime)$ and grid states of ${S}(D)$: Let}
{
\begin{align*}
    e:{I}(D^\prime)\rightarrow {S}(D), \quad {x}\cup \{c\} \mapsto {x}.
\end{align*}}

Then, we linearly extend $e$ to $t\mathbf{C}$.

\begin{prop}
The map $e$ is a filtered quasi-isomorphism.
\end{prop}

\begin{proof}

The same argument works in this case. $e$ can be seen as a quotient of a filtered quasi-isomorphism that is defined in $GC^{-}$. Then, taking quotients of the homotopy equivalences show that $e$ is a filtered quasi-isomorphism.  

\end{proof}

\subsection{Relation with the annular concordance invariant}

\begin{definition}
 For  $0 \leq t=\frac{p}{q} \leq 1$,  Define $t\mathscr{C} = F[V]$ module over grid states $ S(D)$ and \vspace{5mm}\\
 ${\partial}_{t}x := \mathlarger{\sum\limits_{y\in S(D)} \sum\limits_{r \in Rect^{o}(x,y), r \cap \mathbb{X} = \phi}} V^{pO_{U} (r)+ (q-p)O_{L}(r)} y $ for a $x \in S(D)$\\ 
\end{definition}

So $t\mathscr{C}$ is obtained from $t\mathbf{C}$ by setting all $V_{i}'s$ equal to each other.

\begin{prop} \label{imphomtc}
Multiplication by $V_{i}^{q}$ is chain homotopic to multiplication by $V_{j}^{q}$ in $t\mathbf{C}$ if $O_{i}$ and $O_{j}$ belong to the same link component in $L$.
\end{prop}
\begin{proof}
Let $X_{k}$ be the X-marking that is in same row as $O_{m}$ and in the same column as $O_{n}$.
Define $H_{X_{k}}:t\mathbf{C} \rightarrow t\mathbf{C}$,
 \[ H_{X_{k}}(x) := \mathlarger{\sum\limits_{y\in S(D)} \sum\limits_{r \in Rect^{o}(x,y), r \cap \mathbb{X} = {X_{k}}} } V_{1}^{pO_{1} (r)}V_{2}^{pO_{2} (r)}V_{3}^{(q-p)O_{3}(r)}\cdots V_{n-1}^{(q-p)O_{n-1}(r)}V_{1}^{(q-p)O_{n}(r)} y \ \ \ \forall x \in S(D) \]

Then, if $O_{m}$ is one of the markings belonging to $L$ 

\[ \partial H_{X_{k}} + H_{X_{k}} \partial= V_{m}^{q} - V_{n}^{q} \]

It follows that $V_{m}^{q}$ and $V_{n}^{q}$ are chain homotopic. Iterating this argument shows that $V_{i}^{q}$ and $V_{j}^{q}$ are chain homotopic if $O_{i}$ and $O_{j}$ belong to the same link component in $L$.

\end{proof}

In light of Proposition \ref{imphomtc}, we can think of  $t\mathbf{C}$ as a $\mathbb{F}_{2}[V_{i_{1}},\cdots,V_{i_{l}}]$-module. We can also consider the complex $ct\mathbf{C}\cong \frac{t\mathbf{C}}{V_{i_{1}}=\cdots=V_{i_{l}}}$. It can be easily seen that $\mathit{F}_{t}$ filtered quasi-isomorphism type of  $ct\mathbf{C}$ is also an annular link invariant and its homology (denoted by $ctH_{*}(\mathbf{C})$) can be thought of as a $\mathbb{F}_{2}[V]$-module. Now, we are ready to relate $t\mathscr{C}$ with $ct\mathbf{C}$\\

Let $\mathbf{W}_{L}$ be a vector space with two generators, one in $0$ $\mathscr{F}_{t}$-grading and the other in $(1-t)$ $\mathscr{F}_{t}$-grading. Let $\mathbf{W}_{U}$ be a vector space with two generators one in $0$ $\mathscr{F}_{t}$-grading and the other in $t$ $\mathscr{F}_{t}$-grading (Similar to proposition \ref{impvecdef}).
\begin{prop}\label{ctrel}
$t\mathscr{C}$ is quasi-isomorphic to $ct\mathbf{C}\otimes \mathbf{W}_{L}^{n-l-2} \otimes \mathbf{W}_{U}$ where $l$ denotes the number of components in $L$. 
\end{prop}
\begin{proof}

Using Proposition \ref{imphomtc} and a short exact sequence similar to one considered earlier, we can derive the relation.
\end{proof}

\begin{prop}\label{homorel}
$tH(\mathscr{C}) \cong (\mathbb{F}_{2}[V])^{2^{n-1}} \bigoplus Tor$ , where $n$ is the grid number of $U\cup L$.
\end{prop}
\begin{proof}

 Let $\phi: t\mathscr{C} \rightarrow  \frac{t\mathscr{C}}{V-1}$ be the projection onto quotient.\\
Now the quotient, $ \frac{t\mathscr{C}}{V-1} \cong \widetilde{\mathcal{GC}}(- (U\cup L) )$ under the natural identification (we will call the identification map $\chi$). \\  We also know that $ [ \xi ] \in    tH(\mathscr{C}) $ is non-torsion element if and only if $H(\phi)([ \xi ]) \neq 0$. So $ ker(H(\phi))= Tor(t\mathscr{C})$ . \\ Now since $\widetilde{\mathcal{GC}}(-(U \cup L)) \cong \mathbb{F}_{2}^{2^l}$ , it follows that rank of free part of $tH(\mathscr{C})$ is  $2^{n-1}$. Hence we can conclude that $tH(\mathscr{C}) \cong (\mathbb{F}_{2}[V])^{2^{n-1}} \bigoplus Tor$.

\end{proof}

We can also compute the free rank of $ct\mathbf{C}$. 
\begin{prop}
If an annular link $U\cup L$ has $l$ components.
\begin{align*}
    \frac{ctH_{*}(\mathbf{C}(D))}{Tor(ctH_{*}(\mathbf{C}(D))}\cong \mathbb{F}_{2}[V]^{2^{l-1}}.
\end{align*}
\end{prop}
\begin{proof}
The conclusion follows from Proposition \ref{ctrel} and \ref{homorel}.

\end{proof}

The distinguished class $[x^{+}]$ in $ctH_{*}(\mathbf{C})\subset tH(\mathscr{C})$  is a non-torsion element since its image under  $\phi \circ \chi$ is $[x^{NEX}]$ in $\widetilde{\mathcal{GC}}(-(U\cup L) )$, which is non trivial of Maslov grading $1-n$. There is an unique generator of that grading in $\widetilde{\mathcal{GC}}(-(U\cup L) )$. So, there must be a class $[\alpha]$ in $ctH_{*}(\mathbf{C})\subset tH(\mathscr{C})$ for which  $[x^{+}]=V^{k} [\alpha] + [\beta]$ where $[\beta]$ is a torsion class and $k \geq 0$ is maximum. Similarly, it can be seen that $[x^-]$ is also a non-torsion element in the same tower. This gives the following relation with the annular concordance invariant  $\mathcal{A}_{L}(t)$.

\begin{proof} [Proof of Theorem \ref{theorem4}]

Let $\mathcal{F}_{t}$ be the filtration given by $\frac{t}{2}A_{-U} + (1-\frac{t}{2})A_{-L}$ on $\widetilde{\mathcal{GC}}(- (U\cup L) )$. 
 It follows from the definitions that $\mathcal{F}_{t} (x) = - \mathscr{F}_{t} (x) - \frac{t}{2} - (n-l-2)(1-\frac{t}{2})  $  . Since $H(\phi \circ \chi)([\alpha]) \neq 0$. It follows from the definition of $\widetilde{\mathcal{A}_{L}}(t)$ and 4.3 that  $\widetilde{\mathcal{A}_{L}^{U}}(t)=\widetilde{\mathcal{A}_{-L}^{-U}}(t)  \leq  \mathcal{F}_{t} ((\phi \circ \chi)([\alpha])) = - \mathscr{F}_{t}([\alpha]) - \frac{t}{2} - (n-l-2)(1-\frac{t}{2}) $. Therefore, using 4.2 we obtain $\mathcal{A}_{L}(t) \leq - \mathscr{F}_{t}([\alpha])$ . Conversely, if we take an element $ a \neq 0 \in \widetilde{\mathcal{GC}}_{1-n}(-(U\cup L) )$ with  $\widetilde{\mathcal{A}_{L}}(t) =  \mathcal{F}_{t} (a)$ . Since $H(\phi \circ \chi)([\alpha])= [x^{NEX}]=[a]$, it follows that  $\mathcal{A}_{L}(t)= \mathcal{F}_{t}(a) + \frac{t}{2} + (n-l-2)(1-\frac{t}{2}) \geq  \mathcal{F}_{t}([a])  + \frac{t}{2} + (n-l-2)(1-\frac{t}{2})  =  - \mathscr{F}_{t}([\alpha])$. Hence, the equality follows.

\end{proof}

\subsection{Crossing change}

Let $L_+$ and $L_-$ be two annular links with grids $D_+$ and $D_-$ that differ only in a crossing in the $L$ component.
\begin{prop} \label{gencross}
There are $\mathcal{R}$-module maps
\begin{align*}
    C_-\colon ctH(\mathbf{C})(L_+)\rightarrow ctH(\mathbf{C})(L_-) \quad \text{and}\quad  C_+\colon ctH(\mathbf{C})(L_-)\rightarrow ctH(\mathbf{C})(L_+)
\end{align*}
where $C_-$ is homogeneous and preserves the $\mathscr{F}_{t}$ grading, and $C_+$ is homogeneous and shifts $\mathscr{F}_{t}$ degree by $-(1-\frac{t}{2})$. $C_- \circ C_+$ and $C_+ \circ C_-$ are both the multiplication by $V^{q-p}$.
\end{prop}

\begin{figure}\label{crosspic}
    \centering
    \includegraphics[width=\linewidth]{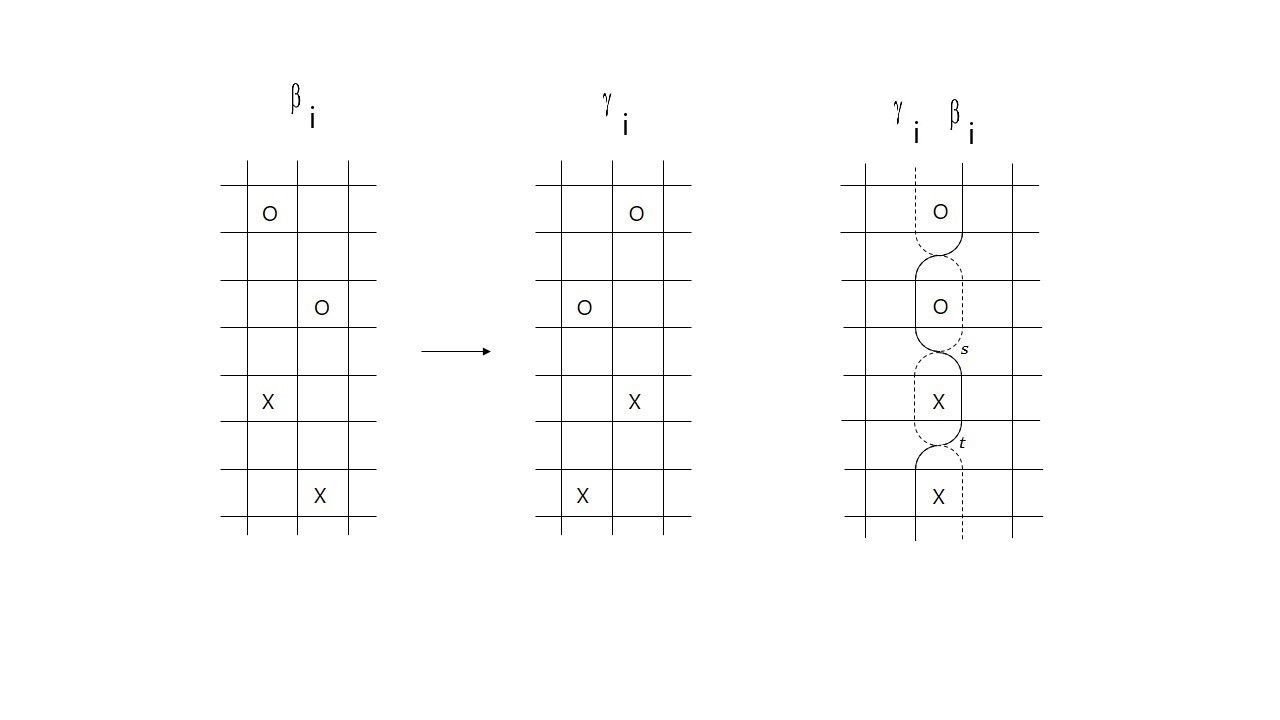}
    \caption{Crossing change move in $t\mathbf{C}$}
\end{figure}
\begin{proof}
{Fix grid states ${x}_+ \in {S}(D_+)$ and ${x}_- \in {S}(D_-)$. We again use the notation $Pent^\circ_s({x}_+, {x}_-)$ for the set of empty pentagons from ${x}_+$ to ${x}_-$ containing $s$ as a vertex, and similarly, $Pent^\circ_s({x}_-, {x}_+)$ for the set of empty pentagons from ${x}_-$ to ${x}_+$ with one vertex at $t$.}

{Consider the $\mathcal{R}$-module maps $c_- \colon ct\mathbf{C}(D_+) \rightarrow ct\mathbf{C}(D_-)$ and $c_+ \colon ct\mathbf{C}(D_-) \rightarrow ct\mathbf{C}(D_+)$ defined on a grid state ${x}_+ \in {S}(D_+)$ and ${x}_- \in {S}(D_-)$ respectively in the following way:}
\begin{align*}
    c_-({x}_+)=\sum\limits_{{y}_-\in {S}(D_-)} \sum\limits_{\Pi \in \text{Pent}^\circ_s({x}_+, {y}_-)}  V_{1}^{pO_{1} (\Pi)}V_{2}^{pO_{2} (\Pi)}V_{3}^{(q-p)O_{3}(\Pi)}\cdots V_{n-1}^{(q-p)O_{n-1}(\Pi)}V_{1}^{(q-p)O_{n}(\Pi)}\cdot {y}_-.
\end{align*}
\begin{align*}
    c_+({x}_-)=\sum\limits_{{y}_+\in {S}(D_+)} \sum\limits_{\Pi \in \text{Pent}^\circ_t({x}_-, {y}_+)}  V_{1}^{pO_{1} (\Pi)}V_{2}^{pO_{2} (\Pi)}V_{3}^{(q-p)O_{3}(\Pi)}\cdots V_{n-1}^{(q-p)O_{n-1}(\Pi)}V_{1}^{(q-p)O_{n}(\Pi)}\cdot {y}_+.
\end{align*}
\end{proof}
\begin{prop}
The map $c_-$ preserves the $\mathscr{F}_{t}$-grading and $c_+$ drops the $\mathscr{F}_{t}$-grading by $(1-\frac{t}{2})$.
\end{prop}
\begin{proof}

The grading changes can be computed by considering local computations for each interval. Let $y$ be the term, appearing in $c_{-}(x)$ and assume there is a left pentagon from $x$ to $y$.

\textbf{Case 1: $y$ is type B} In this case we saw that, $A_{L}(y)-A_{L}(x)=A_{L}(\phi^{-}(y))-A_{L}(x)+ A_{L}(y)- A_{L}(\phi^{-}(y))=0$ and $\mathcal{A}_{U}(y)-\mathcal{A}_{U}(x)=\mathcal{A}_{U}(\phi^{-}(y))-\mathcal{A}_{U}(x)+ \mathcal{A}_{U}(y)- \mathcal{A}_{U}(\phi^{-}(y))=0$. So $\mathscr{F}_{t}(y)-\mathscr{F}_{t}(x)=0$.

\textbf{Case 2: $y$ is type C} In this case, $A_{L}(y)-A_{L}(x)=A_{L}(\phi^{-}(y))-A_{L}(x)+ A_{L}(y)- A_{L}(\phi^{-}(y))= 1-1=0$ and $\mathcal{A}_{U}(y)-\mathcal{A}_{U}(x)=\mathcal{A}_{U}(\phi^{-}(y))-\mathcal{A}_{U}(x)+ \mathcal{A}_{U}(y)- \mathcal{A}_{U}(\phi^{-}(y))=0$. So $\mathscr{F}_{t}(y)-\mathscr{F}_{t}(x)=0$.

For right pentagons, the computation is same except we consider initial corners in $B$ and $C$.

Similarly, we can compute the grading shift for $c_+$.

\end{proof}
\begin{prop}
The maps $c_-$ and $c_+$ are chain maps.
\end{prop}
\begin{proof}

Again the proof is similar to the one given in section \ref{refinesection}.

\end{proof}

The above chain maps $c_-$ and $c_+$ induce the desired maps $C_-$ and $C_+$ on the homologies. In order to verify Proposition \ref{gencross}, we have to show that $C_- \circ C_+$ and $C_+ \circ C_-$ are both the multiplication by $V^{q}$. For this, we can find chain homotopies between the composites $c_- \circ c_+$ respectively $c_+ \circ c_-$ and multiplication by $V^{q-p}$ by taking quotient similar to the proof in the last section.\\

Therefore, we have that $C_- \circ C_+$ and $C_+ \circ C_-$ are both the multiplication by $V^{q-p}$.\\ 

\subsection{Stabilizations}
Now, we will look at stabilizations. First, we will need to define the notion of stabilization of annular links. We define negative stabilization $L^{-}$ to be the annular link obtained from of an annular link $L$  by adding a linked negative crossing [See Figure \ref{aps}]. Similarly, we define positive stabilization $L^{+}$ by generalizing the picture of positive stabilization in Figure \ref{ns}. Obviously, for braids, these correspond to the standard braid stabilizations. We will abuse notations by referring to both grids of annular links $L$ and $L^{+}$ by $L$ and $L^{+}$ respectively. 

\begin{figure}[!tbph]
\begin{center}

\includegraphics[scale=0.15]{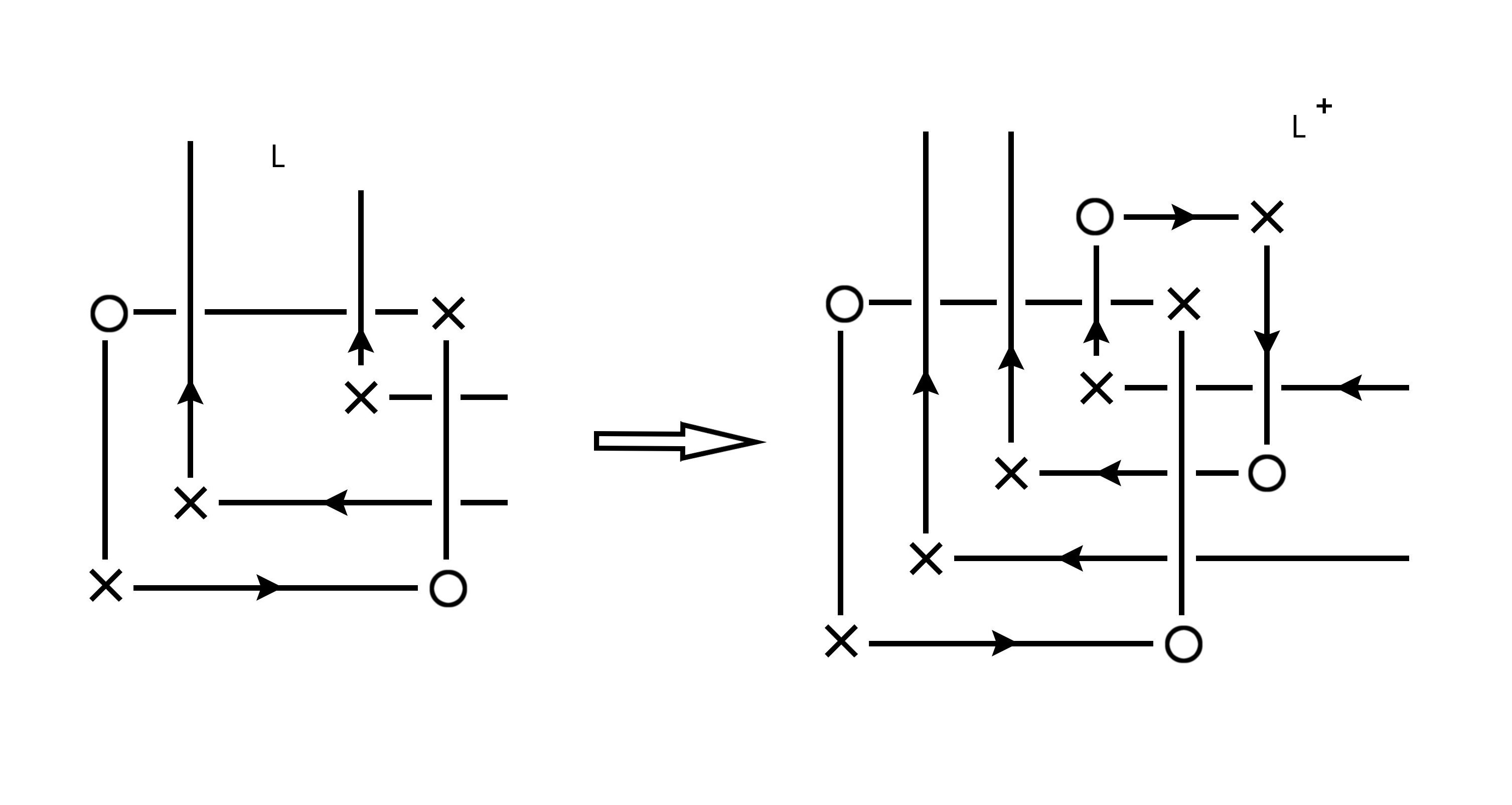}
\caption{Negative stabilization of an annular link}\label{aps}
\end{center}
\end{figure}

We define the $\mathcal{R}$-module maps ${PS}_{-}$ : $ct\mathbf{C}(L^{+}) \rightarrow ct\mathbf{C}(L)$ and ${PS}_{+} : ct\mathbf{C}(L) \rightarrow ct\mathbf{C}(L^{+})$ for a grid state $x \in S(L^{+})$ and $y' \in  S(L)$
respectively in the following way:
\[ {PS}_{-}(x) =\sum_{ y \in S(L)} \sum\limits_{\Pi \in \text{Pent}^\circ_s({x}_+, {y}_-)}  V_{1}^{pO_{1} (\Pi)}V_{2}^{pO_{2} (\Pi)}V_{3}^{(q-p)O_{3}(\Pi)}\cdots V_{n-1}^{(q-p)O_{n-1}(\Pi)}V_{1}^{(q-p)O_{n}(\Pi)}\cdot  y \]

\[ {PS}_{+}(y') =\sum_{ x' \in S(L^{+})} \sum\limits_{\Pi \in \text{Pent}^\circ_t({x}_-, {y}_+)}  V_{1}^{pO_{1} (\Pi)}V_{2}^{pO_{2} (\Pi)}V_{3}^{(q-p)O_{3}(\Pi)}\cdots V_{n-1}^{(q-p)O_{n-1}(\Pi)}V_{1}^{(q-p)O_{n}(\Pi)}\cdot x' \]

For a negative stabilzation $L^{-}$ of $L$, we have the $\mathcal{R}$-module maps ${NS}_{+}$ : $ct\mathbf{C}(L^{-}) \rightarrow ct\mathbf{C}(L)$ and ${NS}_{-} : ct\mathbf{C}(L) \rightarrow ct\mathbf{C}(L^{-})$ for a grid state $x \in S(L^{+})$ and $y' \in  S(L)$ defined as usual.

Now, the following propositions are derived just like the last section.

\begin{prop}
The maps ${NS}_{-}$,  $NS_{+}$, ${PS}_{-}$ and ${PS}_{+}$ are chain maps.
\end{prop}

\begin{prop}
The maps ${PS}_{-}\circ{PS}_{+}$ and ${PS}_{+}\circ{PS}_{-}$ are both homotopic to $V^{q-p}$. The maps ${NS}_{-}\circ{NS}_{+}$ and ${NS}_{+}\circ{NS}_{-}$ are both homotopic to $V^{p}$.

\end{prop}

Now we will compute grading shifts of these maps that will be key for understanding behavior of the invariant under positive/negative stabilizations.

\begin{figure}[!tbph]
\begin{center}
\includegraphics[scale=0.4]{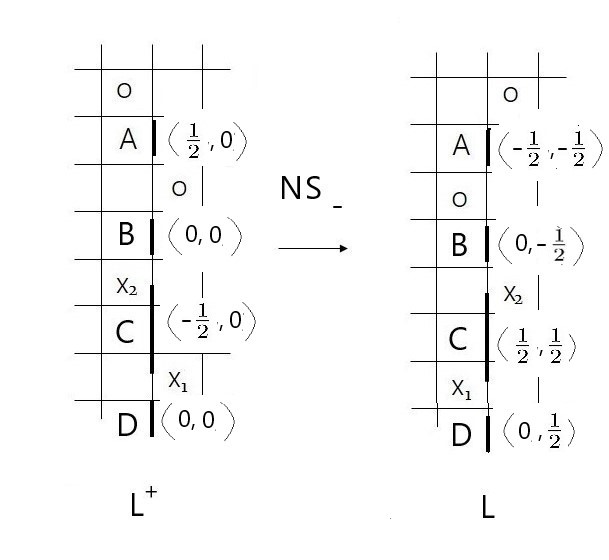}
\caption{Change in local values of $(A, A_{U})$ gradings}\label{newintervals}
\end{center}
\end{figure}

\begin{prop}\label{starchange}
The maps ${NS}_{-}$ shifts $\mathscr{F}_{t}$  by $\frac{1}{2}(1-t)$  and $NS_{+}$ shifts $\mathscr{F}_{t}$  by $-\frac{1}{2}$ . The maps ${PS}_{-}$ shifts $\mathscr{F}_{t}$ by $-\frac{1}{2}$ and $PS_{+}$ shifts $\mathscr{F}_{t}$  by $\frac{t-1}{2}$ .

\end{prop}
\begin{proof}

 We again make local computations for each of the intervals and keep track of local change in both $A=A_{L}+A_{U}$ and $A_{U}$ gradings [See Figure \ref{newintervals}]. We know how $A$ grading changes locally from Lemmma 6.2.1 of \cite{Grid Homology for Knots and Links} and we use the winding number formula (Proposition 11.2.6. in \cite{Grid Homology for Knots and Links}) for computing the $A_{U}$ grading change. If $y$ is a term appearing in ${NS}_{-}(x)$ and there is an empty left pentagon $p$ (pentagon to the the left of vertical circle $\beta_{i}$ or $\gamma_{i}$) from $x$ to $y$ and from Lemma 6.2.1 of \cite{Grid Homology for Knots and Links} that the terminal generator $y$ is either of type $B$ or $C$. So we can use associated left rectangles for the left pentagons to compute grading change.

\textbf{Case 1: $y$ is type B}  Since there are no extra markings in the associated rectangles we have, \[\mathscr{F}_{t}(y)-\mathscr{F}_{t}(x)= \text{ Local change }= \frac{1}{2}(1-\frac{t}{2})+(-\frac{1}{2})\frac{t}{2}=\frac{1-t}{2}. \]

\textbf{Case 2: $y$ is type C} In this case, we have an extra X marking belonging to the unknot in the associated rectangle. So, \[\mathscr{F}_{t}(y)-\mathscr{F}_{t}(x)=\text{ Local change }  - \frac{t}{2}=\frac{1}{2}\cdot\frac{t}{2}+\frac{1}{2}(1-\frac{t}{2})-\frac{t}{2}=\frac{1-t}{2}. \]

Similarly for a right pentagon (pentagon to the the right of vertical circle $\beta_{i}$ or $\gamma_{i}$), we compare it with a right rectangle. Here the initial corner is either of type $B$ or type $C$. In each case, we get grading change $=\frac{1-t}{2}$.\\

Now, let $y'$ be a term appearing in ${NS}_{+}(x')$.\\ 

\textbf{Case 1: $y'$ is type B}
From Figure \ref{explainplus}, we observe that the associated rectangle has one extra X and O markings belonging to $L$ and one less O marking  belonging to the unknot. So we have, \[\mathscr{F}_{t}(y')-\mathscr{F}_{t}(x')=\text{ Local change }- \frac{t}{2}= -\frac{1-t}{2}- \frac{t}{2}=-\frac{1}{2}.\]

\textbf{Case 2: $y'$ is type C}
There are no additional markings. So we have, \[\mathscr{F}_{t}(y')-\mathscr{F}_{t}(x')=\text{ Local change }= -\frac{1}{2}(1-\frac{t}{2}) - \frac{t}{2}\cdot\frac{1}{2}=-\frac{1}{2}.\]

\textbf{Case 3: $y'$ is type D}
There is an additional X marking belonging to $L$. Therefore,
\[\mathscr{F}_{t}(y')-\mathscr{F}_{t}(x')=\text{ Local change } - (1-\frac{t}{2})= \frac{1}{2}(1-\frac{t}{2}) - \frac{1}{2} \cdot \frac{t}{2} - (1-\frac{t}{2})= -\frac{1}{2}.\]

\textbf{Case 4: $y'$ is type A}
Now, the associated rectangle has one extra X marking belonging to $L$ and one less O marking  belonging to the unknot. So we have, 
\[\mathscr{F}_{t}(y')-\mathscr{F}_{t}(x')=\text{ Local change }-(1-\frac{t}{2})-\frac{t}{2}=\frac{1}{2}(1-\frac{t}{2}) + \frac{1}{2} \cdot \frac{t}{2} - 1 = -\frac{1}{2}.\]

Hence, the map $NS_{+}$ is $\mathscr{F}_{t}$ graded of degree $-\frac{1}{2}$. Similarly, we compute the  $PS_{-}$ and $PS_{+}$ shifts.  
\end{proof}

\subsection{Inequalities involving the annular concordance invariant}

\begin{prop}\label{imppropcross}
If the annular links $L_+$ and $L_-$ differ in a crossing change, then for $t \in [0, 2]$ 
\begin{align*}
    \mathcal{A}_{L_-} (t) \leq \mathcal{A}_{L_+} (t) \leq \mathcal{A}_{L_-}(t) + (1-\frac{t}{2})
\end{align*}
and

\begin{align*}
\mathscr{A}_{L_-} (t) \leq \mathscr{A}_{L_+} (t) \leq \mathscr{A}_{L_-}(t) + (1-\frac{t}{2}).
\end{align*}

\end{prop}
\begin{proof}
Consider a non-torsion element $\xi \in ctH(\mathbf{C})(L_-)$ that has grading $-\mathcal{A}_{L_-}(t)$. Since $C_- \circ C_+$ and $C_+ \circ C_-$ are both the multiplication by $V^{q-p}$, $C_+(\xi)$ is non-torsion. The $\mathcal{F}_{t}$ grading of $C_{+}(\xi)$ is $-\mathcal{A}_{L_-} (t)- (1-\frac{t}{2})$. Also since it is in the tower of the distinguished class $[x^+]$, by Theorem \ref{theorem4}  $-\mathcal{A}_{L_+} (t) \geq -\mathcal{A}_{L_-}(t) - (1-\frac{t}{2})$. Similarly, if $\sigma \in ctH(\mathbf{C})(L_+)$ is a non-torsion element with grading $-\mathcal{A}_{L_+}(t)$, then its image $C_-(\eta)$ has grading $-\mathcal{A}_{L_+}(t)$ too. Again since $C_-(\eta)$ is non-torsion and in the tower of the distinguished class, $-\mathcal{A}_{L_+} (t) \leq -\mathcal{A}_{L_-} (t)$. For the second inequality, we use the fact that mirroring takes $\mathcal{A}_{t}$ to $-\mathscr{A}_{t}$.
\end{proof}

\begin{prop}\label{impstabineq}
Let $L^{+}$ and $L^{-}$ denote the positive and negative stabilization of an annular link $L$. Then for $t \in [0, 2]$

\begin{align*}
    \mathcal{A}_{L} (t)- \frac{1}{2}  \leq \mathcal{A}_{L^{-}} (t) \leq \mathcal{A}_{L}(t) - \frac{1-t}{2} 
\end{align*}
and
 
\begin{align*}
    \mathcal{A}_{L} (t)- \frac{1}{2} \leq \mathcal{A}_{L^{+}} (t) \leq \mathcal{A}_{L}(t) + \frac{1-t}{2}
\end{align*}

Also,

\begin{align*}
    \mathscr{A}_{L} (t) - \frac{1-t}{2} \leq \mathscr{A}_{L^{-}} (t) \leq \mathscr{A}_{L}(t) +  \frac{1}{2}
\end{align*}
and

\begin{align*}
    \mathscr{A}_{L} (t) + \frac{1-t}{2}  \leq \mathscr{A}_{L^{+}} (t) \leq \mathscr{A}_{L}(t) +  \frac{1}{2}.
\end{align*}

\end{prop}

\begin{proof}

The proof for $\mathcal{A}_{L}$ inequalities is similar to the last proposition. Instead of using $C_+$ and $C_-$, we use $PS_+$, $PS_-$, $NS_+$
and $NS_-$ to derive the inequalities. Then, we obtain $\mathscr{A}_{L}$ inequalities by taking the mirror.
\end{proof}

Since $\mathscr{A}_{L}(t)= - \mathcal{A}_{m(L)}(t)$, Theorem \ref{theorem4} implies $\mathscr{A}_{L}(t)= \mathscr{F}(t)([\alpha])$ where $[\alpha]$ is the maximum non-torsion class in $[x^+]$ tower in $ct\mathbf{C}(m(L))$. We can use this relationship to derive a slice-Bennequin type inequality. Given a link $L$, a Legendrization of $L$ is a Legendrian link whose topological link type is $L$. A slice-Bennequin type inequality relates classical invariants of the Legendrization with a concordance invariant of the topological link type.

\begin{proof} [Proof of Theorem \ref{theorem5}]
By Theorem \ref{theorem4}, $\mathscr{A}_{\beta}(t) \geq  \mathcal{F}_{t}(x^{+})= \frac{t}{2}A_{U}(x^+) + (1-\frac{t}{2})A_{L}(x^+) = \frac{t}{2} A_{U}(x^+) +(1-\frac{t}{2}) ( A(x^+) -A_{U}(x^+))$.\\
Now, $A_{U}(x^+)=\frac{-1+1+lk(U,L)}{2}= \frac{lk(U,L)}{2}$ using the winding number formula (Equation \ref{windingformula}). Also, $A(x^+)=\frac{ tb(\mathcal{L}\cup U)- rot(\mathcal{L}\cup U) + l+ 1}{2}=\frac{tb(\mathcal{L})- rot(\mathcal{L}) +l+2lk(U,L)}{2}$. Therefore, \[ \mathscr{A}_{L}(t) \geq \frac{lk(U,L)t}{4}+ (1-\frac{t}{2})(\frac{tb(\mathcal{L})- rot(\mathcal{L}) +l+lk(U,L)}{2})\] We also know that $x^-$ is non-torsion in the same tower hence. Hence, $\mathscr{A}_{\beta}(t) \geq  \mathcal{F}_{t}(x^{-})$ and a similar computation shows that \[ \mathscr{A}_{L}(t) \geq \frac{lk(U,L)t}{4}+ (1-\frac{t}{2})(\frac{tb(\mathcal{L})+ rot(\mathcal{L}) +l+lk(U,L)}{2})\].

\end{proof}

The above lower bound is similar in spirit to lower bound given by Plamenevskaya \cite{plam1,plam3} on $\tau$ and Rasmussen's $s$ invariant.

\subsection{ Grid complex of n-cables and  $t\mathscr{C}$ }

 For an annular link $U \cup L$ we will build the $p$-cable by only transform cells in the same row or column of X-markings belonging to $L$.

We will denote the annular $p$-cable generated using this construction by $U_{r} \cup L_{p,q}$. Here, one copy of unknot is replaced by $r$ copies of unknot for any natural number $r$. Let us consider a subset $\mathcal{K}$ of $\mathscr{C}_{U\cup L_{p,q}}$ which is generated by states that contains intermediate north east corners of X in marking inside the block.
\\
\begin{figure}[tbph!]
\begin{center}
\includegraphics[width=.4\textwidth]{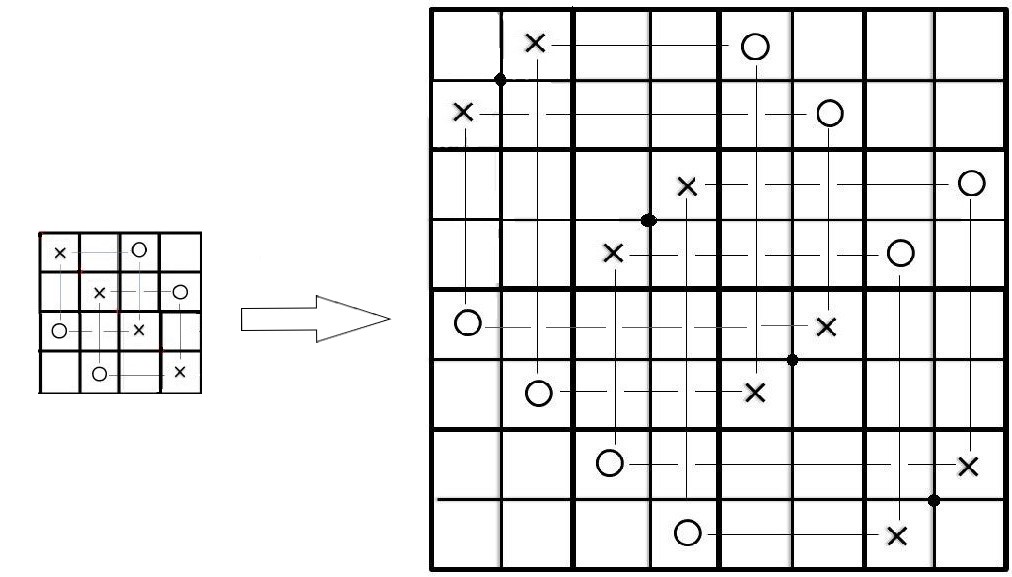}
\end{center}    
\caption{The states in subcomplex $\mathcal{K}$ contains black dots}\label{ablock}
\end{figure}  

The following proposition relates the cable complex with the modified complex.

\begin{prop}
$\mathcal{K}$ of is a  subcomplex of $\mathscr{C}_{U_{r}\cup L_{p,q}}$ and  is isomorphic to $\frac{r}{p}\mathscr{C}$.
\end{prop}
\begin{proof}
There is no rectangle coming out the special points(See Fig ~\ref{ablock}) of $\mathcal{K}$. Therefore, any rectangle coming out of a state in $\mathcal{K}$ must join it with another state in $\mathcal{K}$. Hence, its a subcomplex. We can identify the states of $\mathscr{C}_{U_{r}\cup L_{p,q}}$ with generators of $\frac{r}{p}\mathscr{C}$ . It easily follows from the construction that these two complexes are indeed isomorphic.

\end{proof}

\begin{prop}

There is a chain map $i : \frac{r}{p}\mathscr{C} \rightarrow \mathscr{C}_{U_{r}\cup L_{p,q}}$ such that $i([\alpha])=U^k[\alpha']$ where $k\in \mathbb{N}$ and $[\alpha],[\alpha']$ are non torsion elements in the respective complexes satisfying the described property.
\end{prop}

\begin{proof}

We know that $i$ sends the distinguished state $x^+$ to itself. Since $[x^+] \in  \mathscr{C}_{U_{r}\cup L_{p,q}} = U^m [\alpha']$ for some $m$ and  $[x^+] \in \frac{r}{p}\mathscr{C} = U^n [\alpha]$ for some $n$ , it follows that $i(U^n [\alpha])=i([x^+])= [x^+] = U^m [\alpha'] $. This implies $U^n i([\alpha])= U^m [\alpha']$ . So $m\geq n$ as $[\alpha']$ is top of the non torsion tower in   
$ \mathscr{C}_{U\cup L_{2,q}}$. If $m>n$, then $i([\alpha])=U^k  [\alpha']$ where $k=m-n$ is a natural number. 

\end{proof}

\begin{prop}

$  \tau(L_{p,q} \cup U_{r} ) \geq  p\mathcal{A}_{L}(t) + \frac{(p-1)(p+q-1)}{2} $.

\end{prop}

\begin{proof}

We need to compute the filtered degree of the map $i$ , which is $A(i(x^+)) - p\mathscr{F}_{1/p}(x^+)=A(i(x^+))-A_{U}(x^+)-pA_{L}(x^+) $. So the degree is equal to, $\frac{(p-1)(q-1)}{2} + \frac{(p-1)p}{2}= \frac{(p-1)(p+q-1)}{2} $. Therefore, $  \tau(L_{p,q} \cup U) \geq p\mathcal{A}_{L}(t) + \frac{(p-1)(p+q-1)}{2} $.

\end{proof}

\begin{prop}\label{quasi1}
 If $L \cup U$ is a quasi-positive link. then $\mathscr{A}_{L}(t)= -\mathcal{A}_{m(L)}(t)=\mathscr{F}_{t}(x^+)$ where $x^+$ is the distinguished generator in the grid of the annular link $m(L)$.
\end{prop}

\begin{proof}

We will show that if  $L$ is a quasi-positive then $[x^+]$ is the top of the non-torsion tower in $ctH_{*}(\mathbf{C})(U\cup m(L))$. First, notice that if $L \cup U$ is quasi-positive, then the $n$-cable $U_{r} \cup L_{n,q} $ is quasi-positive for $q\geq 0$. It follows $[x^+]$ is the top of the non-torsion tower in $\mathscr{C}_{U_{r}\cup m(L_{n,q})}$ since $\mathscr{C}_{U_{r}\cup m(L_{n,q})}=GC^{-}\otimes W^{\otimes N}$ for some $N$ and  $[x^+]$ is the top of the tower in $GC^{-}(L)$ for a quasi-positive link $L$ \cite{cavallo2}. Now since we have $i(x^+)=x^+$, it follows that $x^+$ is the top of the tower in $ct\mathbf{C}$. Therefore, $\mathcal{A}_{m(L)}(t)=-\mathscr{F}_{t}(x^+)$.
\end{proof}

\pagebreak

\section{Braided cobordisms}

We can define an invariant of a braid $\beta$ by considering annular invariant of the annular link $U \cup \beta $ where $U$ acts like the braid axis.

\subsection{Properties of the annular invariant for braids}

A braided cobordism $\Sigma$ from $\beta_{1} = \Sigma \cap (\mathbb{S}^3 \times \{ 0 \})$ to
$\beta_{2} = \Sigma \cap (\mathbb{S}^3 \times \{ 1 \})$ is braid-orientable if it admits an orientation compatible with the braid-like orientations of $\beta_{1}$ and $\beta_{2}$.

\begin{prop}\label{bcob}
 If $\beta_{1}$ and $\beta_{2}$  are braids , and $\Sigma$
is a braid-orientable braided cobordism from $\beta_{2}$ to $\beta_{1}$ with $s$ split saddles and $d$ deaths then

\[ \mathscr{A}_{\beta_{1}}(t)-\mathscr{A}_{\beta_{2}}(t) \leq  (s-d) (1-\frac{t}{2}).\]
\end{prop}

\begin{proof}
From Section 5.1 we know that split moves have filtered degree $1-\frac{t}{2}$ and death moves have filtered degree $-1+\frac{t}{2}$. We get the inequality by adding the contributions in the cobordism.
\end{proof}

Now we will study the effect of braid stabilization on the annular invariant. 

\begin{prop}
If $\beta$ is a $n$-braid,let $\beta^{+}$ and $\beta^{-}$ represent the $n+1$-braids obtained by positively and negatively stabilizing $\beta$ respectively. Then, we have the following inequalities

\begin{align*}
    \mathscr{A}_{\beta} (t) - \frac{1-t}{2} \leq \mathscr{A}_{\beta^{-}} (t) \leq \mathscr{A}_{\beta}(t) +  \frac{1}{2}
\end{align*}
and

\begin{align*}
    \mathscr{A}_{\beta} (t) + \frac{1-t}{2}  \leq \mathscr{A}_{\beta^{+}} (t) \leq \mathscr{A}_{\beta}(t) +  \frac{1}{2}.
\end{align*}

\end{prop}
\begin{proof}

The above inequalities follow directly from Proposition \ref{impstabineq}.

\end{proof}

\begin{prop}

$\mathscr{A}_{\beta}(2)=\frac{n}{2}$ for any $n$-braid $\beta$.

\end{prop}

\begin{proof}
To see this, we can consider a strong braided cobordism from $\beta$ to $Id_{n}$. Then, it follows from Proposition \ref{bcob}.

\end{proof}

\begin{prop}

If $\beta$ has $1$ component then, $\mathscr{A}_{\beta}(0)= \tau(\beta)+\frac{n}{2}$.

\end{prop}

\begin{proof}

We can think of $-\mathscr{A}_{\beta}(0)$ as the max $A_{\beta}$ grading of the $x^+$ tower in $0\mathbf{C}(m(\beta))\cong HFK^{-,2}(m(\beta))$. Now from the inclusion isomorphism in \cite{LOSS}, it is clear that the non-torsion tower is taken to the non-torsion tower in  $HFK^{-}(m(\beta))$ and the grading shift is $-\frac{n}{2}$. The conclusion follows.  
\end{proof}

\subsection{Bounds on band rank} 

We also get the following lower bound on band rank from the annular invariant.

\begin{theorem}\label{theorem6}
Let $\beta$ be an $n$-braid with $l$ components and ${Id}_{n}$ be the identity $n$-braid. Then
$ \mathscr{A}_{{\beta}}(t)-\mathscr{A}_{{Id}_{n}}(t) \leq \frac{rk_{n}(\beta)+l-n}{2} (1-\frac{t}{2})$.

\end{theorem}
\begin{proof} [Proof of Theorem \ref{theorem6}]

Recall that if a braid $\beta$ has band rank $rk_{n}(\beta)$, then it can be written as

\begin{equation*}
  \beta= \prod_{j=1}^{rk_{n}(\beta)} \omega_{j} \sigma_{i_{j}} ^{\pm 1} \omega_{j}^{-1}
\end{equation*}

Therefore, there is a cobordism from ${Id}_{n}$ to $\beta$ that has $rk_{n}$ number of saddles. Now, $rk_{n}=s+m$, where $s$ is the number of split and $m$ is the number of merge cobordism componenets in that cobordism. Also we have, $s-m=l-n$. Hence, the inequality follows from Proposition \ref{bcob}. 
\end{proof}

$K \subset S$ is called a ribbon knot if it bounds a smoothly embedded disk in $B^4$, Morse, with no interior maxima. Rudolph's theorem (\cite{lrudolph}) tells us, if $K$ is ribbon then it has a closed  $n$-braid representative $\sigma$ with $rk_{n}(\sigma)= n-1$. So if a closed  $n$-braid representative $\beta$ of a some slice knot $K$ satisfies  $| \mathscr{A}_{{\beta}}(t)-\mathscr{A}_{{Id}_{n}}(t) | >  (n-1) (1-\frac{t}{2})$ and this inequality is preserved under (de)stabilization then that will provide a counterexample to slice-ribbon conjecture.\\

Given a braid $\beta$ we can transform it to a Legendrian link (See Figure \ref{LegBraid})which is call Legendrization of the braid. The following inequality establishes an interesting relationship between band rank and classical Legendrian invariants. 

\begin{figure}  
\begin{center}
\includegraphics[width=0.70\textwidth]{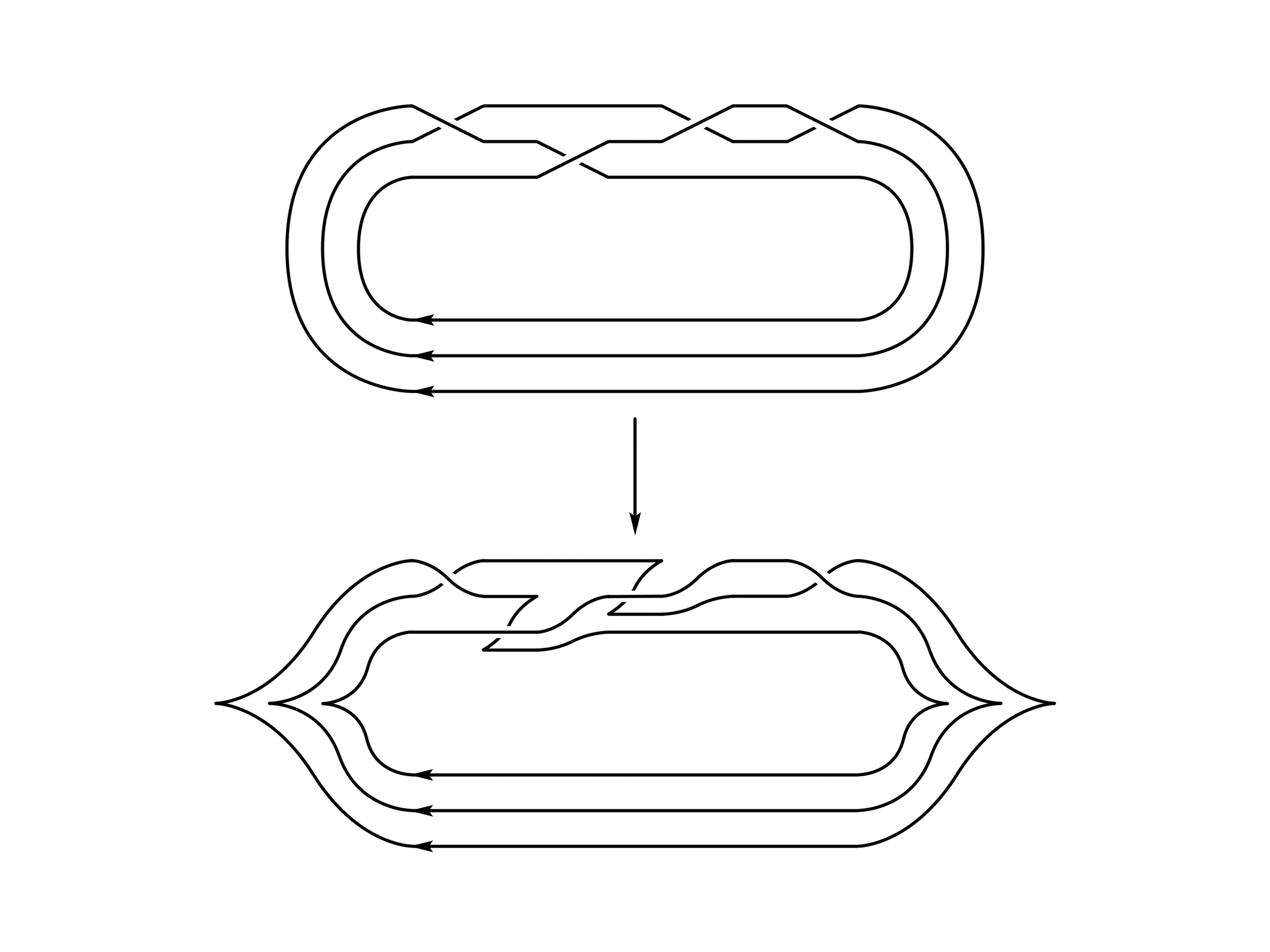}    
\caption{Legendrization of a braid}
\label{LegBraid}
\end{center}
\end{figure}

\begin{proof}  [Proof of Theorem \ref{theorem7}]

For all $t \in [0,2]$, the inequality 
\[ \frac{rk_{n}(\beta)+l-n}{2} (1-\frac{t}{2}) \geq \frac{nt}{4}+ (1-\frac{t}{2})\frac{tb(\mathcal{L})+ |rot(\mathcal{L})|+l+n}{2} - \frac{n}{2} \] 
follows easily by combining Theorem \ref{theorem6} and Theorem \ref{theorem5}. In particular for $t=0$, we get \[ \frac{rk_{n}(\beta)+l-n}{2} \geq \frac{ tb(\mathcal{L})+ |rot(\mathcal{L})|+l}{2} .\] 

\end{proof}

\subsection{Right veering and transverse properties}

The annular braid invariant also has right veering and transverse properties properties analagous to Grigsby-Wehrli-Licata invariant. Before proving those results, let us first compute the invariant for quasi-positive braids.
\begin{prop} \label{generalquasi}
If $\beta$ is a quasi-positive braid of index $n$ with $l$ componenents. Then, $\mathscr{A}_{\beta}(t)=  t \frac{-wr(\beta)-l+n}{4}+\frac{wr(\beta)+l}{2}$.
\end{prop}
\begin{proof}
By Proposition \ref{quasi1}, $\mathscr{A}_{\beta}(t) =  \mathcal{F}_{t}(x^{+})= \frac{t}{2}A_{U}(x^+) + (1-\frac{t}{2})A_{L}(x^+) = (-1+t) A_{U}(x^+) +(1-\frac{t}{2}) A(x^+)$.\\
Now $A_{U}(x^+)=\frac{-1+1+lk(U,L)}{2}= \frac{n}{2}$ and $A(x^+)=\frac{ -(n+1) + (wr(\beta)+2n) + l+ 1}{2} = \frac{wr(\beta)+n+l}{2}$. Hence, $\mathcal{A}_{\beta}(t)= (t-1) \frac{n}{2}+ \frac{wr(\beta)+n+l}{2} (1-\frac{t}{2})$. 
\end{proof}

For quasi-positive braids with one component, we recover the $\tau$ to be $ \frac{sl(\beta)+1}{2}$ since $\mathscr{A}_{\beta}(0)=\tau(\beta)+\frac{n}{2}$.\\

Recall that we defined slope function $m_t(\beta)$ and $y$-value $y_t(\beta)$ associated to $\mathscr{A}_{\beta}(t)$. Let us define a related function.  
\begin{definition}
For a braid $\beta$ and $t \in [0,2]$

\[M_t(\beta):=2m_t(\beta)+y_t(\beta).\]

\end{definition}

Then by Proposition \ref{slopeprop},  $M_t(\beta)=A_{U}(x_{0})$ for some generator $x_{0}$ at each $t \in [0,2]$. The function $M_t$ also has an alternative formulation in the $ct\mathbf{C}$ complex. By Theorem \ref{theorem4}, $\mathcal{A}_{L}(t) = - \mathscr{F}_{t}([\alpha])$ where $\alpha$ is a maximum non-torsion element in $[x^+]$ tower in $ct\mathbf{C}$. Then again by Proposition \ref{slopeprop}, $M_t(\beta)=-A_{U}(\alpha)$.\\ 

Notice that $M_t(\beta)=\frac{n}{2}$ for all $t$ in the case of quasi-positive braids. The following proposition gives a more general criteria in terms of right veering.

\begin{prop}
If $M_t(\beta)=\frac{n}{2}$ for some $0 < t <2$ then $\beta$ is right veering.
\end{prop}

\begin{proof}

Suppose $\beta$ is non-right veering. Then we know that $\hat{\theta}$ and $\tilde{\theta}$ vanishes.  Consider the short exact sequence \\
\begin{tikzcd}
 0 \arrow{r} & t\mathscr{C}(\beta) \arrow{r}{V} & t\mathscr{C}(\beta) \arrow{r} & \widetilde{GC}(m(\beta)) \arrow{r}{p} & 0 
\end{tikzcd}\\ 
In the induced long exact sequence, $p_{*}$ takes $[x^+]$ to $\tilde{\theta}$. It follows that $[x^+]$ is $V$-image in $t\mathscr{C}$. But this implies $M_t(\beta) < A_{U}(x^+)=\frac{n}{2}$.

\end{proof}

\begin{prop}

If $M_{t_{0}}(\beta)=\frac{n}{2}$ for some $t_{0} \in [0,2)$ then $M_{t}(\beta)=\frac{n}{2}$ for $2 \geq t>t_{0}$.
\end{prop}

\begin{proof}

By Proposition \ref{slopeprop}, we can easily see that For $t>t_0$ $\Delta M_t \geq 0$. Since $\frac{n}{2}$ is the maximum possible value, it follows that  $M_{t}(\beta)=\frac{n}{2}$ for $2 \geq t>t_{0}$.  
\end{proof}

\begin{prop}
Suppose $\beta_{+}$ and $\beta_{-}$ are obtained from $\beta$ by addition of a positive and negative crossing respectively then
\[ M_{\beta_{-}}(t) \geq M_{\beta}(t) \geq M_{\beta_{+}}(t) .\]

\end{prop}

\begin{proof}

It is easy to see that $m(\beta)$ can be obtained from $m(\beta_{-})$ by addition of a negative crossing. So we can consider the crossing change map $c_{-}:ct\mathbf{C}(m(\beta_{-})) \rightarrow ct\mathbf{C}(m(\beta))$. Now let $\alpha \in t\mathbf{C}(m(\beta_{-}))$ non-torsion in $x^+$ tower with $A_{U}(\alpha)= - M_{\beta_{-}}(t)$. Then $c_{-}(\alpha)$ is also non-torsion in the $x^+$ tower with $A_{U}(c_{-}(\alpha))= - M_{\beta_{-}}(t)$. Hence, it follows that $-M_{\beta}(t) \geq - M_{\beta_{-}}(t)$. Similarly, we obtain the other inequality. 

\end{proof}

Let \[ \mathscr{M}_t := \{ \beta \ | \ \beta \text{ has index } n, M_t(\beta)=\frac{n}{2} \} .\]

Let $t_{1}, t_{2} \cdots t_{n}$ be real numbers satisfying $ 0 \leq t_{1} \leq \cdots \leq t_{n} < 2$. Then we clearly have, \[ QP \subseteq \mathscr{M}_{t_{1}} \subseteq \cdots \subseteq \mathscr{M}_{t_{n}} \subseteq RV . \] Where $QP$ and $RV$ denotes the monoids of Quasi-positive and right-veering braids respectively. \\

\begin{proof} [Proof of Theorem \ref{theorem8}]

 Suppose $\beta \in \mathscr{M}_t$, then there is a non-torsion element $\alpha$ in the $x^+$ tower in $ct\mathbf{C}(m(\beta))$ with $A_{U}(\alpha)=-\frac{n}{2}$. Now, we consider the negative stabilization map ${NS}_{-}:ct\mathbf{C}(m(\beta)) \rightarrow ct\mathbf{C}(m(\beta_{+stab}))$. Then, from Proposition \ref{starchange} we have, $A_{U}({NS}_{-} (\alpha))=-\frac{n}{2}-\frac{1}{2}=-\frac{n+1}{2}$. This implies that max non torsion element in $t\mathbf{C}(m(\beta_{+stab}))$ also has $A_{U}$ equal to $-\frac{n+1}{2}$ since its the minimum possible value. Therefore, $\beta_{+stab} \in \mathscr{M}_t$. \\
   
Let us take an index $N$- braid $\beta_1\in \mathscr{M}_t$ and an index $M$- braid $\beta_2\in \mathscr{M}_t$. To prove that $\mathscr{M}_t $ is a monoid, we need to show that $\beta_1\beta_2 \in \mathscr{M}_t$. Firs, we observe that  ${M}_t(\beta_1 \sqcup \beta2)= {M}_t(\beta_1)+{M}_t(\beta_2)=\frac{N}{2}+\frac{M}{2}=\frac{M+N}{2}$. So $\beta_1 \sqcup \beta_2 \in \mathscr{M}_t$. Now Baldwin \cite{Bald1} showed that $\beta_1\beta_2$ is transversely isotopic to $\beta_1 \sqcup \beta_2$ after adding negative crossing. Since $M_t$ is non decreasing for addition of negative crossing and membership in $\mathscr{M}_t$ is a transverse invariant, it follows that $\beta_1\beta_2\in \mathscr{M}_t$.

\end{proof}

\pagebreak

\end{document}